\documentclass{amsart}
\usepackage{graphicx}
\usepackage{psfrag}
\usepackage[all]{xy}
\usepackage{amscd}
\numberwithin{equation}{section}
\usepackage{amssymb}

\let\cal\mathcal
\def\Ascr{{\cal A}}
\def\Bscr{{\cal B}}
\def\Cscr{{\cal C}}

\def\Escr{{\cal E}}
\def\Fscr{{\cal F}}

\def\Hscr{{\cal H}}

\def\Lscr{{\cal L}}
\def\Mscr{{\cal M}}
\def\Nscr{{\cal N}}
\def\Oscr{{\cal O}}

\def\Tscr{{\cal T}}
\def\Uscr{{\cal U}}
\def\Vscr{{\cal V}}

\let\blb\mathbb
\def\CC{{\blb C}}

\def \AA{{\blb A}}
\def \ZZ{{\blb Z}}

\def \NN{{\blb N}}
\def \RR{{\blb R}}
\def \HH{{\blb H}}
\def \HC{{\operatorname{HH}}}
\def \DGMod{{\operatorname{DGMod}}}

\def\Sh{\operatorname{Sh}}

\def\id{\text{id}}
\def\Id{\operatorname{id}}

\def\Der{\operatorname{Der}}

\def\ctimes{\mathbin{\hat{\otimes}}}

\def\Mod{\operatorname{Mod}}
\def\mod{\operatorname{mod}}

\def\gr{\operatorname{gr}}

\def\gr{\operatorname {gr}}
\def\Spec{\operatorname {Spec}}

\def\Ext{\operatorname {Ext}}
\def\Hom{\operatorname {Hom}}

\def\End{\operatorname {End}}

\def\Tr{\operatorname {Tr}}

\def\ker{\operatorname {ker}}

\def\End{\operatorname {End}}

\def\id{{\operatorname {id}}}

\def\r{\rightarrow}
\def\l{\leftarrow}

\DeclareMathOperator{\HEnd}{\mathcal{E}\mathit{nd}}

\DeclareMathOperator{\td}{td}
\DeclareMathOperator{\Td}{td}

\let\invlim\projlim

\newtheorem{lemma}{Lemma}[section]
\newtheorem{proposition}[lemma]{Proposition}
\newtheorem{theorem}[lemma]{Theorem}
\newtheorem{corollary}[lemma]{Corollary}

\newtheorem{lemmas}{Lemma}[subsection]
\newtheorem{propositions}[lemmas]{Proposition}
\newtheorem{theorems}[lemmas]{Theorem}

\theoremstyle{definition}

\newtheorem{example}[lemma]{Example}

\newtheorem{examples}[lemmas]{Example}
\newtheorem{definitions}[lemmas]{Definition}

{

}

\theoremstyle{remark}

\newtheorem{remark}[lemma]{Remark}
\newtheorem{remarks}[lemmas]{Remark}

\newdimen\uboxsep \uboxsep=1ex
\def\uboxn#1{\vtop to 0pt{\hrule height 0pt depth 0pt\vskip\uboxsep
\hbox to 0pt{\hss #1\hss}\vss}}

\def\uboxs#1{\vbox to 0pt{\vss\hbox to 0pt{\hss #1\hss}
\vskip\uboxsep\hrule height 0pt depth 0pt}}

\def\Gl{\operatorname{Gl}}
\def\poly{\operatorname{poly}}
\def\Hoch{\operatorname{Hoch}}
\def\coord{\operatorname{coord}}
\def\aff{\operatorname{aff}}

\def\poly{\operatorname{poly}}

\DeclareMathOperator{\HDer}{\mathcal{D}\mathit{er}}
\author{Damien Calaque} 
\address{Universit\'e de Lyon \\ Universit\'e Lyon 1 \\ Institut Camille Jordan CNRS UMR 5208 \\
43 boulevard du 11 novembre 1918 \\ F-69622 Villeurbanne Cedex \\ France}
\email{calaque@math.univ-lyon1.fr}
\author{Michel Van den Bergh}
\address{Departement WNI\\Universiteit Hasselt\\ Universitaire Campus\\ Building D\\ 3590 
Diepenbeek\\ Belgium} 
\thanks{The second author is a director of research at the FWO} 
\thanks{The results of this paper were partially obtained while the
second author was visiting the Universit\'e Claude Bernard at Lyon. He hereby
thanks the latter for its kind hospitality.
}
 \email{michel.vandenbergh@uhasselt.be} 

\title{Hochschild cohomology and Atiyah classes}

\begin{document}

\maketitle

\begin{abstract}
  In this paper we prove that on a smooth algebraic
 variety the HKR-morphism twisted by
  the square root of the Todd genus gives an isomorphism between the
  sheaf of poly-vector fields and the sheaf of poly-differential
  operators, both considered as derived Gerstenhaber algebras. In
  particular we obtain an isomorphism between Hochschild cohomology
  and the cohomology of poly-vector fields which is compatible with
  the Lie bracket and the cupproduct. The latter compatibility is an
  unpublished result by Kontsevich.

  Our proof is set in the framework of Lie algebroids and so applies
  without modification in much more general settings as well.
\end{abstract}

\setcounter{tocdepth}{1}

\tableofcontents

\section{Introduction}

In the body of the paper we will work over a ringed site over a field
of charactristic zero~$k$. Thus our results are for example applicable to 
stacks.  However in this
introduction we will state our results for a commutatively ringed space
$(X,\Oscr_X)$.

By definition a \emph{Lie algebroid} on $X$ is a sheaf of Lie algebras
$\Lscr$ which is an $\Oscr_X$-module and is equipped with an action
$\Lscr\times \Oscr_X\r \Oscr_X$ with properties mimicking those of the
tangent bundle (see \S\ref{ref-3.2-9} for a more precise definition).  Throughout $\Lscr$ will be a locally free Lie
algebroid over $(X,\Oscr_X)$ of constant rank~$d$.

The
advantage of the Lie algebroid framework is that it allows one to
treat the algebraic/complex analytic and $C^\infty$-case in a uniform way. 
\begin{example}
  The following are examples of (locally free) Lie algebroids.
\begin{enumerate}
\item The sheaf of vector fields on a $C^\infty$-manifold.
\item The sheaf of holomorphic vector fields on a complex analytic variety.
\item The sheaf of algebraic vector fields on a smooth algebraic variety.
\item $\Oscr_X\otimes \frak{g}$ where $\frak{g}$ is the Lie algebra
of an algebraic group acting on a smooth algebraic variety $X$. 
\end{enumerate}
\end{example}
\begin{example} Assume that $X$ is an affine integral singular algebraic variety. Then
$\Tscr_X$ is not locally free. However it is always possible to construct
a locally free sub Lie algebroid $\Lscr\subset\Tscr_X$. So our setting
applies to some extent to the singular case as well. 
\end{example}
The Atiyah class $A(\Lscr)$ of $\Lscr$ is 
the element of $\Ext^1_{X}(\Lscr,\Lscr^\ast\otimes_X\Lscr)$ which is the
obstruction against the existence of an $\Lscr$-connection on $\Lscr$.
The $i$'th ($i>0$) scalar Atiyah class $a_i(\Lscr)$ is defined as
\[
\operatorname{Alt}\Tr(A(\Lscr)^i)\in H^i(X,(\wedge^i \Lscr)^\ast)
\]
In the $C^\infty$ or affine case we have $a_i(\Lscr)=0$ as the
cohomology groups $ H^i(X,(\wedge^i \Lscr)^\ast)$ vanish.  If $X$ is a
Kahler manifold and $\Lscr$ is the sheaf of holomorphic vector fields
then $a_i(\Tscr_X)$ coincides with the $i$'th Chern class of $\Tscr_X$
(see e.g.\ \cite[(1.4)]{Kapranov2}).

The Todd class of $\Lscr$ is defined as
\[
\td(\Lscr)=\det(q(A(\Lscr)))
\]
where
\[
q(x)=\frac{x}{1-e^{-x}}
\]
%
%
One sees without difficulty that $\td(\Lscr)$ can be
expanded formally in terms of $a_i(\Lscr)$.  

\medskip

The sheaf of $\Lscr$-poly-vector fields on $X$ is defined as
$T^\Lscr_{\poly}(\Oscr_X)=\oplus_i \wedge^i \Lscr$. This agrees with the  
standard definitions in case $\Lscr$ is one of the
variants of the tangent bundle described in the examples above. It is
easy to prove that $T^\Lscr_{\poly}(\Oscr_X)$ is a sheaf of Gerstenhaber algebras
on~$X$. 

In the case that $X$ is a $C^\infty$-manifold Kontsevich introduced
the sheaf of so-called poly-differential operators on $X$. This is basically
a localized version of the Hochschild complex.\footnote{It is not
entirely trivial to make the Hochschild complex into a (pre)sheaf
as the assignment $U\mapsto C^\ast(\Gamma(U,\Oscr_U))$ is not compatible
with restriction.}
It is
straightforward to construct a Lie algebroid generalization
$D^{\Lscr}_{\poly}(\Oscr_X)$ of this concept as well (see \cite{cal} or
\S\ref{ref-3.2.2-11}). Like $T^\Lscr_{\poly}(\Oscr_X)$, $D^{\Lscr}_{\poly}(\Oscr_X)$ is equipped
with a Lie bracket and an associative cupproduct but these operations
satisfy the Gerstenhaber axioms only up to globally defined
homotopies (see e.g.\ \cite{VG}). 

The so-called Hochschild-Kostand-Rosenberg map is a 
quasi-isomorphism between $T^\Lscr_{\poly}(\Oscr_X)$ and
$D^\Lscr_{\poly}(\Oscr_X)$ \cite{cal,ye2}. This paper is concerned with
the failure of the HKR-map to be compatible with the Lie brackets and
cupproducts on $T^\Lscr_{\poly}(\Oscr_X)$ and
$D^\Lscr_{\poly}(\Oscr_X)$.

\medskip

Let $D(X)$ be the derived category of sheaves of $k$-vector spaces. This
category is equipped with a symmetric monoidal structure given by the
derived tensor product. As indicated above
$D^{\Lscr}_{\poly}(X)$ is a Gerstenhaber algebra in $D(X)$. 
We have the following result
\begin{theorem} (see \S\ref{ref-9-114}) \label{ref-1.3-0}
 The map in
  $D(X)$
\begin{equation}
\label{ref-1.1-1}
T^\Lscr_{\poly}(\Oscr_X)\xrightarrow{\operatorname{HKR}\circ (\td(\Lscr)^{1/2}\wedge-)}
D^\Lscr_{\poly}(\Oscr_X)
\end{equation}
is an isomorphism of Gerstenhaber algebras in $D(X)$.
\end{theorem}
Applying the hypercohomology functor $\HH^\ast(X,-)$ we get
\begin{corollary} \label{ref-1.4-2} The map
\begin{equation}
\label{ref-1.2-3}
\bigoplus_{i,j}H^j(X,\wedge^i \Lscr)\xrightarrow{\operatorname{HKR}\circ (\td(\Lscr)^{1/2}\wedge-)} 
\HH^\ast (X,D^\Lscr_{\poly}(\Oscr_X))
\end{equation}
is an isomorphism of Gerstenhaber algebras. 
\end{corollary}
Let us restrict to the setting where $X$ is a smooth algebraic variety
and $\Lscr=\Tscr_X$.  In that case it follows from the proof of
\cite[Thm 3.1(1)]{VdB35} together with \cite[Thm 7.5.1]{lowenvdb2} that the
righthand side of \eqref{ref-1.2-3} can be viewed as the Hochschild
cohomology $\HC^\ast(X)$  of $X$ (in the sense that it controls for example the deformation theory of $\Mod(\Oscr_X)$). So we may rephraze Corollary \ref{ref-1.4-2}
as
\begin{corollary}
\label{ref-1.5-4}
There is an isomorphism of Gerstenhaber algebras
\begin{equation}
\label{kontsevich}
\bigoplus_{i,j}H^j(X,\wedge^i\Tscr_X)\xrightarrow{\operatorname{HKR}\circ (\td(\Lscr)^{1/2}\wedge-)}
\operatorname{HH}^\ast(X)
\end{equation}
\end{corollary} 
A version of this result which refers only to the cupproduct was
proved by Kontsevich (see \cite[Thm 5.1]{Caldararu2}). 
For the cupproduct one can use Swan's definition of Hochschild cohomology
\cite{Swan}
\begin{equation}
\label{ref-1.3-5}
\operatorname{HH}^i(X)=\Ext^i_{X\times X}(\Oscr_X,\Oscr_X)
\end{equation}
as Yekutieli \cite{ye6,ye2} shows that there is an isomorphism
\begin{equation}
\label{ref-1.4-6}
\HH^i(X,D_{\poly}(\Oscr_X))\r \Ext^i_{X\times X}(\Oscr_X,\Oscr_X) 
\end{equation}
which is compatible with the cuppproduct on the left and the Yoneda
product on the right.

\begin{remark}
The algebra isomorphism \eqref{kontsevich} is part of a more general
conjecture by Caldararu \cite{Caldararu2} which
involves also the Hochschild \emph{homology} of $X$. Other parts of
this conjecture were proved by Markarian and Ramadoss \cite{Markarian,Ramadoss}.\end{remark}
\begin{remark} The cupproduct on $T^\Lscr_{\poly}(\Oscr_X)$ and
  $D^\Lscr_{\poly}(\Oscr_X)$ is $\Oscr_X$-linear and hence these objects can
  also be considered as algebras in $D(\Mod(\Oscr_X))$. Likewise the
  map \eqref{ref-1.1-1} can be viewed as an isomorphism in
  $D(\Mod(X))$.

The cupproduct on $T^\Lscr_{\poly}(\Oscr_X)$ is commutative and the
cupproduct on $D^\Lscr_{\poly}(\Oscr_X)$ is commutative up to a
homotopy given by the bullet product \cite{VG}. However the latter is
\emph{not} $\Oscr_X$-linear.  Hence $D^\Lscr_{\poly}(\Oscr_X)$ is not
commutative in $D(\Mod(\Oscr_X))$ and thus Theorem \ref{ref-1.3-0}
does \emph{not} hold in $D(\Mod(\Oscr_X))$ even if we consider only
the cupproduct.  

This situation is reminiscent of the Duflo isomorphism
$S\frak{g}\r U\frak{g}$ which only becomes an algebra isomorphism
after taking invariants.  The analogue of taking invariants in our
setting is taking global sections. 

We thank Andrei Caldararu
  for bringing this point to our attention.
\end{remark}

If we look only at the Lie algebra structure we actually prove a
result which is somewhat stronger than Theorem \ref{ref-1.3-0}.  Let
$\operatorname{HoLieAlg}(X)$ be the category of sheaves of DG-Lie
algebras with quasi-isomorphisms inverted.
\begin{theorem} \label{ref-1.6-7} (see \S\ref{ref-6.4-88})
  \label{ref-1.6-8} The  isomorphism \eqref{ref-1.1-1} between
  $T^\Lscr_{\poly}(X)$ and $D^{\Lscr}_{\poly}(X)$ is obtained from
an isomorphism in
  $\operatorname{HoLieAlg}(X)$.
\end{theorem}
This theorem can be considered as a generalization of global formality
results in \cite{cdh,CF,Dolgushev,Ko3,VdB35,ye3}. Global formality on
the sheaf level is important for deformation theory.  See for example
\cite{caha,Ko10,leje,VdB35, ye3}.

\medskip

The isomorphism in $\operatorname{HoLieAlg}(X)$ is obtained by
globalizing a local formality isomorphism \cite{Ko3,Tamarkin}.
If we take Kontsevich's local formality isomorphism then we obtain
compatibility with cupproduct in Theorem \ref{ref-1.3-0}
from the compatibility of the local
formality isomorphism with tangent cohomology \cite{Ko3,MT,Mochizuki}.
Kontsevich's local formality isomorphism is only defined when the
ground field contains $\RR$ but this is not a problem since we show
that it is sufficient to prove Theorem \ref{ref-1.3-0} over a suitable
extension of the base field. 

\medskip


\medskip

An alternative approach to Theorem \ref{ref-1.3-0} could be to work
directly in the setting of $G_\infty$-algebras.  Unfortunately it is
unknown if Kontsevich's $L_\infty$-morphism can be lifted to a
$G_\infty$-morphism. In \cite{CalVdB} we will use Tamarkin's local
$G_\infty$-formality isomorphism to construct a
$G_\infty$-quasi-isomorphism between $T^\Lscr_{\poly}(X)$ and
$D^\Lscr_{\poly}(X)$. In the case that $\Lscr$ is a tangent bundle
this was proved recently in \cite{dtt} using very different methods.
Like in \cite{dtt} we are unfortunately not able to write down
the resulting isomorphism on hypercohomology. Thus in
this way we obtain a result which is less precise than Corollary
\ref{ref-1.5-4}.  This is why we have decided to publish the current
paper separately.

\section{Acknowledgement} 
This paper is hugely in debt to Kontsevich's fundamental work on formality.
In particular without the many deep results and insights contained in \cite{Ko3} this
paper could not have been written. 

Our proof of the global formality result Theorem \ref{ref-1.6-7} follows
 the general outline of \cite{VdB35} which in turn
was heavily inspired by \cite{ye3}. We use in an essential way an
algebraic version of formal geometry.  Algebraic versions of formal
geometry were introduced independently and around the same time by
Bezrukavikov and Kaledin in \cite{BezKal} and Yekutieli in
\cite{ye3}. The language we use is closer to \cite{ye3}. As a result
various technical statements can be traced back in some form to \cite{ye3}. 

We wish to thank Andrei Caldararu, Vasiliy Dolgushev, Charles Torossian and Amnon
Yekutieli for useful conversations and comments.
\section{Notations and conventions}

Throughout this paper $k$ is a field of characteristic zero.
Unadorned tensorproducts are over $k$. 

Many of the objects we use are equipped with some
kind of topology, but if an object is introduced without a specified
topology we assume that it is equipped with the discrete topology.

If an object carries a natural grading then all constructions associated
to it are implicitly performed in the graded context.  This applies
in particular to completions. 

Since all our constructions are natural in the sense that they do not
depend on any choices we work mostly with rings and modules instead of
with sheaves since this often simplifies the notations. However we
freely sheafify such construction if needed.

On a double (or higher) complex we use the Koszul sign convention with
respect to total degree. 

\section{Preliminaries}

\subsection{Categories of vector spaces}

Below we will work with various enhanced symmetric monoidal categories
of $k$-vector spaces.  Which category we work in will usually be clear
from the context but in order to be precise we list here the various
possibilities.

\subsubsection{Complete topological vector spaces}

For us a complete topological vector space $V$ will be a topological
vector space whose topology is generated by a separated, exhaustive
descending filtration $V=F_0V\supset F_1V\supset\cdots$. This filtration
is however not considered as part of the structure.

The completed
tensor product 
\[
V\ctimes W=\projlim_p (V\otimes W/(F_p V\otimes W+V\otimes F_p W))
\]
makes the category of complete topological vector spaces
into a symmetric monoidal category.

\subsubsection{Filtered complete topological vector spaces}

A filtered complete topological vector space is by definition a
topological vector space $V$, equipped with an ascending separated,
exhaustive filtration $F^m V$ (which is considered part of the
structure) such that each $F^m V$ is a complete topological vector
space and the inclusion maps $F^m V\hookrightarrow F^{m+1} V$ are continuous. 

The (completed) tensor product of two filtered complete topological
vector space $V$ and $W$ is defined by
\[
F^m(V\ctimes W)=\sum_{p+q=m} F^pV\ctimes F^q W
\]
(where the summation sign refers of course to convergent sums). 

\subsubsection{Graded filtered complete topological vector spaces}

Graded filtered complete topological vector spaces are the most
general objects we will encounter below. These are simply graded
objects over the linear category of filtered complete topological
vector spaces.  The DG-Lie algebra of poly-differential operators of
$k[[t_1,\ldots,t_d]]$ (see below) is naturally a DG-Lie algebra over
the category of graded filtered complete topological vector spaces.

\subsection{Lie algebroids}\label{ref-3.2-9}

Below $R$ is a commutative $k$-algebra
and $L$ is a {\em Lie algebroid} over $R$ which is free of rank $d$.
Namely, $L$ is a Lie $k$-algebra equipped with an $R$-module structure and 
a Lie algebra map $\rho:L\to\Der(R)$ such that 
$[l_1,rl_2]=r[l_1,l_2]+\rho(l_1)(r)l_2$ for $l_1,l_2\in L$ and $r\in R$. 
$\rho$ is called the {\it anchor map} and we usually suppress it
from the notations writing $l(r)$ instead of $\rho(l)(r)$ ($l\in L$, $r\in R$). 
In particular, $R\oplus L$ becomes a Lie algebra with bracket given by 
$[(r,l),(r',l')]=(l(r')-l'(r),[l,l'])$. 

Associated to $L$ there are various constructions which are analogous
to constructions occurring for enveloping algebras and rings of
differential operators.  In the next few paragraphs we fix some notations
for them and recall the properties we need. For more information the
reader is referred to \cite{cal,cdh,nt,xu}.

\subsubsection{The enveloping algebra of a Lie algebroid}
\label{ref-3.2.1-10}

Let $UL$ be the enveloping algebra associated to $L$. It is the
quotient of the enveloping algebra associated to the Lie algebra
$R\oplus L$ by the following relations: $r\otimes l=rl$ ($r\in R$,
$l\in R\oplus L$). If we want to emphasize $R$ then we write $U_R L$.
$UL$ has a canonical filtration obtained by respectively assigning
length $0$ and $1$ to elements of $R$ and $L$. We equip $UL$ with the
left $R$-module structure given by the natural embedding $R\r UL$ and 
we view $UL$ as an $R$-bimodule with the same left and right
structure. For this bimodule structure $UL$ is a cocommutative
$R$-coring in the sense that there is a natural cocommutative
coassociative comultiplication $\Delta:UL\r UL\otimes_{R} UL$ and
counit $\epsilon:UL\r R$. Assuming the Sweedler convention the
comultiplication is defined by
\begin{align*}
\Delta(f)&=f\otimes 1 & &\text{for $f\in R$}\\
\Delta(l)&=l\otimes 1+1 \otimes l&&\text{for $l\in L$}\\
\Delta(DE)&=D_{(1)}E_{(1)}\otimes D_{(2)}E_{(2)} &&\text{for $D,E\in UL$}
\end{align*}
Note that it requires some verification to show that this
is well defined. To do this note that $UL\otimes_R UL$ is a right 
$UL\otimes UL$-module in the obvious way. One proves inductively 
on the length of $D$,
expressed as a product of elements of $L$, that in $UL\otimes_R UL$ one has 
\[
(D_{(1)}\otimes D_{(2)})(f\otimes 1-1\otimes f)=0
\]
for $f$ in $R$. It follows immediately that if $E'\otimes E''\in UL\otimes_R UL$
then
\[
(D_{(1)}\otimes D_{(2)})\cdot (E'\otimes E'')\overset{\text{def}}{=}
D_{(1)}E'\otimes D_{(2)}E''
\]
is well defined. 

There is a unique way to extend the anchor map into an algebra morphism $\rho:U(L)\to \End(R)$. 
As before we write $D(r)$ instead of $\rho(D)(r)$ ($D\in UL$, $r\in R$), and then the counit 
on $UL$ is given by
\[
\epsilon(D)=D(1)\,.
\]
$UL$ is a so-called ``Hopf algebroid with anchor'' \cite{xu}.  As we
are in the cocommutative case this is expressed by the property
\[
D_{(1)}(f) D_{(2)}=Df \qquad \text{$(f\in R, D\in UL)$}\,.
\]

\subsubsection{$L$-poly-vector fields and $L$-poly-differential operators}\label{ref-3.2.2-11}

$T_{\poly}^L(R)$ is the Lie algebra of $L$-poly-vector fields
\cite{cal}. I.e.\ it is the graded vector space $\wedge_R (L)[1]$
equipped with the graded Lie bracket obtained by extending the Lie
bracket on $L$. We equip $T_{\poly}^L(R)$ with the standard
cupproduct (which is of degree one with our shifted grading).  In this way
$T_{\poly}^L(R)$ becomes a (shifted) Gerstenhaber algebra. 

$D_{\poly}^L(R)$ is the DG-Lie algebra of $L$-poly-differential operators
\cite{cal}. I.e.\ it is the graded vector space $T_R(UL)[1]$ equipped with
the natural structure of a DG-Lie algebra \cite{cal}. The  Lie bracket on $D_{\poly}^L(R)$ given 
by $[D_1,D_2]_G=D_1\bullet D_2-(-1)^{|D_1||D_2|} D_2\bullet D_1$, where 
\[
D_1\bullet D_2=\sum_{i=0}^{|D_1|} 
(-1)^{i|D_2|}(\id^{\otimes i}\otimes \Delta^{|D_2|}
\otimes \id^{\otimes |D_1|-i})(D_1)
\cdot (1^{\otimes i}\otimes D_2\otimes 1^{\otimes |D_1|-i})\,.
\]
Let $m=1\otimes 1\in D^L_{\poly}(R)_1$. Then 
the differential $d$ on $D_{\poly}^L(R)$ is given by 
\[
d(D)=[m,-]\,. 
\]
For the cupproduct we use the sign-modification by
Gerstenhaber-Voronov \cite{VG}.  This sign-modification is necessary
to make the cohomology of $D_{\poly}^L(R)$ into a (shifted) Gerstenhaber
algebra.  We put
\[
D_1\cup D_2=(-1)^{|D_1||D_2|}D_1\otimes D_2
\]

There is a HKR-theorem relating $T^L_{\poly}(R)$ and $D_{\poly}^L(R)$ \cite{cal}. Namely the map
\begin{equation}\label{ref-3.1-12}
\mu:l_1\wedge\cdots \wedge l_n\mapsto (-1)^{n(n-1)/2}\frac{1}{n!}
\sum_{\sigma\in S_{n+1}}\epsilon(\sigma)l_{\sigma(1)}\otimes\cdots\otimes l_{\sigma(n)}
\end{equation}
defines a quasi-isomorphism between $(T^L_{\poly}(R),0)$ and $(D_{\poly}^L(R),d)$
which induces an isomorphism of shifted Gerstenhaber algebras on cohomology. 
Note
that for this last fact to be true one needs the unconventional sign in \eqref{ref-3.1-12}.

\subsubsection{Algebraic functoriality}

The formation of $UL$, $T_{\poly}^L(R)$ and $D_{\poly}^L(R)$ depends
functorially on $L$ in a suitable sense.  
\begin{definitions} \label{algebraicmorphism}
An {\it algebraic morphism} of Lie algebroids 
\[
(R,L)\longrightarrow (T,M)
\]
is a pair $(\varrho,\ell)$ of an algebra morphism $\varrho:R\to T$ and a Lie algebra morphism 
$\ell:L\to M$ such that for any $r\in R$ and any $l\in L$, 
$$
\varrho\big(l(r)\big)=\ell(l)(\varrho(r))\quad\textrm{and}\quad\ell(rl)=\varrho(r)\ell(l)\,.
$$
\end{definitions}
For any algebraic morphism $(R,L)\to(T,M)$ there are obvious associated maps 
\begin{align*}
U_RL&\longrightarrow U_TM\,,\\
T_{\poly}^L(R)&\longrightarrow T_{\poly}^M(T)\,,\\
D_{\poly}^L(R)&\longrightarrow D_{\poly}^M(T)\,,
\end{align*}
that are compatible with all algebraic structures. 

\subsubsection{Pairings, the De Rham complex and $L$-connections}
Put $L^\ast=\Hom_R(L,R)$. We identify $\wedge_R^nL^\ast$ with the $R$-dual of $\wedge_R^n L$ 
via the pairing
\begin{equation}
\label{ref-3.2-13}
(\sigma_1\wedge \cdots\wedge \sigma_n,
l_1\wedge \cdots\wedge l_n)=\det \sigma_i(l_j)\,.
\end{equation}
If $\tau\in L^\ast$ then we denote contraction by $\tau$ acting on
$\wedge^n _RL$ by $\tau\wedge-$. I.e.
\[
\tau \wedge (l_1\wedge \cdots\wedge l_n)=\sum_i (-1)^{i-1}  \tau(l_i)( l_1\wedge \cdots
\wedge\hat{l}_i\wedge \cdots \wedge l_n)
\]
We make $\wedge_R L$ into a $\wedge_R L^\ast$-module by extending the $-\wedge-$-action.
I.e.
\[
(\tau_1\wedge\cdots \wedge \tau_m)\wedge (l_1\wedge \cdots\wedge l_n)=
\tau_1\wedge (\tau_2\wedge (\cdots  \wedge(\tau_m\wedge(l_1\wedge \cdots\wedge l_n))\cdots))
\]
An easy verification shows
\begin{equation}
\label{eqadjoint}
(\sigma_1\wedge \cdots\wedge \sigma_n,
l_1\wedge \cdots\wedge l_n)=\langle\sigma_{m+1}\wedge \cdots \wedge\sigma_n,
(\sigma_m\wedge\cdots\wedge\sigma_1)\wedge (l_1\wedge \cdots\wedge l_n)\rangle
\end{equation}
The Lie algebroid analogue for the De Rham complex is a DG-algebra which as
graded algebra is equal to $\wedge_R L^\ast$. 
With the identification \eqref{ref-3.2-13} the differential on $\wedge_R L^\ast$ is given
by the usual formula for differential forms \cite[Prop 2.25(f)]{warner}: 
\begin{equation}\label{ref-3.3-14}
d\omega(l_0,\ldots,l_n)=
\sum_{i=0}^n (-1)^i l_i(\omega(l_0,\ldots,\hat{l}_i,\ldots, l_p))
+\sum_{i<j} (-1)^{i+j} \omega([l_i,l_j],l_0,\ldots,\hat{l}_i,\ldots, \hat{l}_j,\ldots,l_n)\,.
\end{equation}
In other words the anchor map $\rho:L\r \Der_k(R)=\Hom_R(\Omega^1_R,R)$ dualizes to a morphism
of DG-algebras
\begin{equation}
\label{ref-3.4-15}
\rho^\ast:\Omega_R\r \wedge_R L^\ast\,.
\end{equation}
From \eqref{ref-3.3-14} we deduce in particular $(df)(l)=l(f)$ for
$f\in R$, $l\in L$, and if $(l_i)_i$ is an $R$-basis of $L$ then
\[
d(l_k^\ast)(l_i,l_j)=-l_k^\ast([l_i,l_j])\,.
\]
In other words $(\wedge_R L^\ast,d)$ completely encodes the Lie-algebroid 
structure of $L$. 

\begin{remarks} In the literature a
  morphism between Lie algebroids $(R,L)\to(T,M)$ is usually defined
  as a morphism of DG-algebras
  $\eta:\wedge_TM^\ast\longrightarrow\wedge_RL^\ast$.  See e.g.\
  \cite{cattaneo}.  One could call such morphisms ``geometric'' to
  differentiate them from the algebraic ones we use.  We have already
  encountered one geometric morphism, namely \eqref{ref-3.4-15}.
\end{remarks}

If $M$ is an $R$-module then a $L$-connection on $M$ is a map 
$L\otimes M\r M: l\otimes m\mapsto \nabla_l(m)$ with the following
properties: for $l,l_1,l_2\in L$, $m\in M$, $f\in R$ we have
\begin{gather*}
\nabla_l(fm)=l(f)m+f\nabla_l(m)\,, \\
\nabla_{fl}(m)=f\nabla_l(m)\,. 
\end{gather*}
The connection is flat if in addition we have
\[
[\nabla_{l_1},\nabla_{l_2}]=\nabla_{[l_1,l_2]}\,.
\]
In that case $M$ automatically becomes a left $UL$-module. 
Moreover, if $(l_i)_i$ is a basis of $L$ then we put a left $\wedge_R L^\ast$-DG-module
structure on $\wedge_R L^\ast\otimes_R M $ by defining the differential as
\begin{equation}
\label{ref-3.5-16}
\nabla(\omega\otimes m)=d\omega\otimes m+\sum_i l_i^\ast\omega\otimes \nabla_{l_i}(m)\,.
\end{equation}
 Recall that if $C$ is 
a commutative $R$-DG-algebra then a {\it flat connection} on $M$ is a derivation of square zero on 
$C\otimes_RM$ of degree one which makes $C\otimes_RM$ into a DG-$C$-module\footnote{We recall that 
if $M$ is in a category of complete topological vector spaces then one has to write $C\ctimes M$ instead. }.  Thus $\nabla$ is a flat $\wedge_RL^\ast$-connection on $M$. 

\subsubsection{$L$-jets}\label{ref-3.2.6-17}

Let $(UL)_{\le n}$ be the elements of degree $\le n$ for the canonical filtration
on $UL$ introduced in \S\ref{ref-3.2.1-10}. The $L$-$n$-jets
are defined as 
\[
J^nL=\Hom_R((UL)_{\le n},R)
\]
(this is unambiguous, as the left and right $R$-modules structures on
$UL$ are the same, see \S\ref{ref-3.2.1-10}). We also
put
\begin{equation}
\label{ref-3.6-18}
JL=\invlim_n J^nL.
\end{equation}
We now formulate some properties of $JL$. Most of these properties 
hold for $J^nL$ as well. $JL$ has a natural commutative
algebra structure obtained from the comultiplication on $UL$. Thus for
$\phi_1,\phi_2\in JL$, $D\in UL$ we have
\[
(\phi_1\phi_2)(D)=\phi_1(D_{(1)})\phi_2(D_{(2)})\,,
\]
and the unit in $JL$ is given by the counit on $UL$. It is well-known 
that $JL$ has a lot of extra structure which we now elucidate. First 
of all there are two distinct monomorphisms of $k$-algebras 
\begin{align*}
\alpha_1:&R\r JL:r\mapsto (D\mapsto r\epsilon(D))\,,\\
\alpha_2:&R\r JL:r\mapsto (D\mapsto D(r))\,.
\end{align*}
It will be convenient to write $R_i=\alpha_i(R)$
and to view $JL$ as an $R_1-R_2$-bimodule. 

Define $\epsilon:JL\r R$ by $\epsilon(\phi)=\phi(1)$ and put
$J^cL=\ker \epsilon$. It is easy to see that $\epsilon\circ\alpha_1=
\epsilon\circ\alpha_2=\Id_R$. We conclude that
\begin{equation}
\label{ref-3.7-19}
JL=R_1\oplus J^cL=R_2\oplus J^c L
\end{equation}
The filtration on $JL$ induced by \eqref{ref-3.6-18} coincides with the
$J^cL$-adic filtration.  If we filter $JL$ with the $J^cL$-adic
filtration then we obtain
\begin{equation}
\label{ref-3.8-20}
\gr JL=S_{R} L^\ast
\end{equation}
and the $R_1$ and $R_2$-action on the r.h.s. of this equation coincide (here
 and below the letter $S$ stands for ``symmetric algebra''). 

There are also two different commuting actions by derivations
of $L$ on $JL$. Let $l\in L$, $\phi\in JL$, $D\in UL$.
\begin{align*}
{}^1\nabla_l(\phi)(D)&=l(\phi(D))-\phi(lD)\\
{}^2\nabla_l(\phi)(D)&=\phi(Dl)
\end{align*}
Again it will be convenient to write $L_i$ for $L$ acting 
by ${}^i\nabla$. Then ${}^i\nabla$ defines
a flat $L_i$-connection on $JL$, considered as an $R_i$-module. Thus
$JL$ becomes a $UL_1-UL_2$-bimodule (with both $UL_1$ and $UL_2$ acting
on the left). For some of the verifications below we note that the $UL_2$ action on $JL$ takes
the simple form
\[
(D\cdot \phi)(E)= \phi(ED)
\]
(for $D,E\in UL_2$, $\phi \in JL$). 

The induced actions on $\gr JL=S_R L^\ast$ of $l\in L$, considered as an
element of $L_1$ and $L_2$, are given by the contractions $i_{-l}$ and
$i_l$.

\begin{examples} In case $R$ is the coordinate ring of a smooth
affine algebraic variety and $L=\Der_k(R)$ then we may identify $JL$
with the completion of $R\otimes R$ at the kernel of the multiplication
map $R\otimes R\r R$. The two action of $R$ on $JL$ are respectively
$R\otimes 1$ and $1\otimes R$. 

Similarly a derivation on $R$ can be extended to $R\otimes R$ in two
ways by letting it act respectively on the first and second factor. Since
derivations are continuous they act on adic completions and hence
in particular on $JL$.  This provides the two actions of $L$ on $JL$. 
\end{examples}

As ${}^1\nabla_l$ acts by derivation on $JL$ is its easy to see that
the resulting $\wedge_{R_1} L_1^\ast$-DG-module $(\wedge_{R_1}
L_1^\ast\otimes_{R_1} JL,{}^1\nabla)$ (see \eqref{ref-3.5-16}) is
actually a commutative $\wedge_{R_1} L_1^\ast$-DG-algebra. The
following result is well-known.
\begin{propositions}\label{ref-3.2.3-21}
The inclusion $\alpha_2:R\hookrightarrow JL$ defines a quasi-isomorphism
$$
R\longrightarrow\wedge_{R_1} L_1^\ast\otimes_{R_1} JL\,;\,r\longmapsto 1\otimes\alpha_2(r)\,.
$$
\end{propositions}
\begin{proof}
It is easy to see that if $r\in R_2$ then 
${}^1\nabla(1\otimes r)=0$. To prove that we obtain a quasi-isomorphism 
we filter $JL$ by the $J^cL$-adic filtration. We obtain 
the following associated graded complex 
\begin{equation}\label{ref-3.9-22}
0\r R\r S_R L^{\ast} \r L^\ast\otimes_R S_R L^{\ast}
\r \cdots \r \wedge_R^d L^\ast\otimes S_R L^\ast\r 0
\end{equation}
where the differential is given by $-\sum_j l_j^\ast\otimes i_{l_j}$ 
for a basis $(l_j)_j$ of $L$. It is easy to see that \eqref{ref-3.9-22} is exact.
\end{proof}

\subsection{Relative poly-vector fields, poly-differential operators and forms}\label{ref-3.3-23}

\subsubsection{Definitions}

We need relative poly-differential operators and poly-vector fields. So
assume that $A\r B$ is a morphism of commutative DG-$k$-algebras. Then
\begin{align*}
T_{\poly,A}(B)&=\bigoplus_n T^n_{\poly,A}(B)\\
D_{\poly,A}(B)&=\bigoplus_n D^n_{\poly,A}(B)
\end{align*}
where $T^n_{\poly,A}(B)=\bigwedge^{n+1}_B\Der_A(B)$.  Similarly
$D^n_{\poly,A}(B)$ is the set of maps with $n+1$ arguments
$B\otimes_A\cdots \otimes_A B\r B$ which are differential operators
when we equip $B$ with the diagonal $B\otimes_A\cdots \otimes_A
B$-algebra structure. 

If we consider $A$ and $B$ just as algebras then $T_{\poly,A}(B)$ and
$D_{\poly,A}(B)$ are DG-Lie algebras in the usual way. The
differential on $T_{\poly,A}(B)$ is trivial and the differential on
$D_{\poly,A}(B)$ is the restriction of the Hochschild differential. We
denote it by $d_{\Hoch}$. The differential $d_B$ on $B$ induces 
a differential $[d_B,-]$ on $T^n_{\poly,A}(B)$, $D^n_{\poly,A}(B)$ which
commutes with $d_{\Hoch}$ on the latter. The total differentials on
$T^n_{\poly,A}(B)$ and $D^n_{\poly,A}(B)$ are respectively $[d_B,-]$ and
$[d_B,-]+d_{\Hoch}$. 

Recall that one can also consider the DG-algebra $\Omega_{B/A}$ of relative 
differentials\footnote{Here and in the rest of the paper the symbol ``$\Omega$'' 
is used in the sense of continous differentials (since we usually deal with complete 
topological vector spaces). See \cite[\S 5.4]{VdB35} for a more detailed discussion of 
continuous differentials in a slightly restricted case. There are no surprizes. }. 
The differential is $d_B+d_{{\rm DR}}$. 
There is a contraction map between relative one-differentials and relative vector fields: 
it is defined as the $B$-linear map 
$$
\Omega^1_{B/A}\otimes{\rm Der}_A(B)\r B:fdg\otimes\xi\mapsto f\xi(g)\,.
$$

\subsubsection{Relation with $L$-jets}

 Let us begin by the observation that 
\begin{equation}\label{eq-3.3.0}
L_2\r \Der_{R_1}(JL):l\mapsto (\theta\mapsto {}^2\nabla_l(\theta))
\end{equation}
is a Lie algebra morphism.\footnote{Together with $R_2\to JL$ 
this actually is an algebraic morphism of Lie algebroids. } 
\begin{lemmas}\label{ref-3.3.1-24}
\eqref{eq-3.3.0} yields a well-defined isomorphism of $JL$-modules 
\begin{equation}
\label{eqlem331}
JL\otimes_{R_2} L_2\r \Der_{R_1}(JL):\phi\otimes l\mapsto (\theta\mapsto \phi\cdot {}^2\nabla_l(\theta))\,.
\end{equation}
\end{lemmas} 
\begin{proof}
It suffices to check that this is the case for the associated graded modules for 
the $J^cL$-adic filtration, which is easy.
\end{proof}

Let $D_{R_1}(JL)$ be the ring of differential operators of $JL$ relative to $R_1$, 
considered as a $R_2$-module. Since the $L_2$-action on $JL$ commutes with the $R_1$-action 
we obtain a ring homomorphism
\begin{equation}\label{ref-3.10-25}
UL_2\r D_{R_1}(JL):D\mapsto (\theta\mapsto D(\theta))\,.
\end{equation}
It is easy to check that together with $R_2\to JL$ this gives a Hopf algebroid homomorphism 
$(R_2,UL_2)\to(JL,D_{R_1}(JL))$. 
\begin{lemmas}\label{ref-3.3.2-26}
\eqref{ref-3.10-25} yields a well-defined isomorphism of $JL$-modules. 
\begin{equation}\label{ref-3.11-27}
JL\ctimes_{R_2} UL_2\r D_{R_1}(JL):\phi\otimes D\mapsto (\theta\mapsto \phi D(\theta))\,.
\end{equation}
\end{lemmas}
\begin{proof}
It is easily verified that this map is well-defined. To prove that it is an isomorphism 
we extend the natural filtration on $UL_2$ to a filtration on the l.h.s. of \eqref{ref-3.11-27} 
and we filter the r.h.s. by order of differential operators. We then obtain a map
\[
JL\ctimes_{R_2} SL_2=S_{JL} (JL\otimes_{R_2} L_2)\r S_{JL}(\Der_{R_1}(JL))
\]
which is induced from the natural map $JL\otimes_{R_2} L_2\r \Der_{R_1}(JL)$.
This map is an isomorphism by Lemma \ref{ref-3.3.1-24}. 
\end{proof}
Using again that the $L_2$-action on $JL$ commutes with the $R_1$-action
we obtain natural DG-Lie algebra morphisms 
\begin{equation}\label{ref-3.12-28}
\begin{split}
T^{L_2}_{\poly}(R_2)&\r T_{\poly, R_1}(JL)\,,\\
D^{L_2}_{\poly}(R_2)&\r D_{\poly, R_1}(JL)\,.
\end{split}
\end{equation}
We obtain the following (see \cite[Lemma 5.1(22)]{ye3}):
\begin{lemmas}
The maps \eqref{ref-3.12-28} induce well-defined isomorphisms of $JL$-modules. 
\begin{align*}
JL\ctimes_{R_2} T^{L_2}_{\poly}(R_2)&\r T_{\poly, R_1}(JL)\\
JL\ctimes_{R_2} D^{L_2}_{\poly}(R_2)&\r D_{\poly, R_1}(JL)\,.
\end{align*}
\end{lemmas}
\begin{proof}
It easy to check that these maps are well-defined. As an example we prove that the 
second map is an isomorphism. We have isomorphisms of vector spaces 
\begin{align*}
D_{\poly, R_1}(JL)&=T_{JL}(D_{R_1}(JL))[1]=T_{JL}(JL\ctimes_{R_2} UL_2 )[1]\\
&=JL\ctimes_{R_2} T_{R_2}(UL_2)[1]=JL\ctimes_{R_2} D_{\poly}^{L_2}(R_2)\,.
\end{align*}
In the first line we have used Lemma \ref{ref-3.3.2-26}. One now easily shows 
that the resulting isomorphism $JL\ctimes_{R_2} D_{\poly}^{L_2}(R_2)\cong D_{\poly, R_1}(JL)$ 
is indeed the morphism given in the statement of the lemma. 
\end{proof}
We also have: 
\begin{lemmas}\label{ref-3.3.4-29}
Let $C$ be a commutative $R_1$-DG-algebra. The canonical maps
\begin{equation}\label{ref-3.13-30}
\begin{split}
(C\ctimes_{R_1} JL) \ctimes_{R_2} T_{\poly}^{L_2}(R_2)&\r 
T_{\poly,C}(C\ctimes_{R_1} JL)\\
(C\ctimes_{R_1} JL) \ctimes_{R_2} D_{\poly}^{L_2}(R_2)&\r 
D_{\poly,C}(C\ctimes_{R_1} JL)
\end{split}
\end{equation}
obtained by linearly extending the canonical maps
\begin{align*}
T_{\poly}^{L_2}(R_2)&\r T_{\poly,C}(C\ctimes_{R_1} JL)\\
D_{\poly}^{L_2}(R_2)&\r D_{\poly,C}(C\ctimes_{R_1} JL)
\end{align*}
are well-defined isomorphisms. If $JL$ carries a flat $C$-connection $\nabla$
which is compatible with the $R_2$-action on $C\ctimes_{R_1} JL$
then $\nabla\otimes \id$
on the left of \eqref{ref-3.13-30} corresponds to $[\nabla,-]$ on the right. 
\end{lemmas}
\begin{proof}
We restrict ourselves to the case of poly-differential operators. 
The case of poly-vector fields is similar. 
Using the fact that $JL$ is formally smooth and formally of finite type over $R$ we
easily deduce that the canonical map
\[
C\ctimes_{R_1} D_{\poly,R_1}(JL)\r D_{\poly,C}(C\ctimes_{R_1} JL)
\]
is an isomorphism. Combining this with Lemma \ref{ref-3.3.4-29} yields the
required isomorphism. It is easily seen that this isomorphism yields
the asserted compatibility for flat $C$-connections. 
\end{proof}

\subsubsection{Differentials and $L$-jets}\label{sec-formsandjets}

Let us introduce a $JL$-linear map 
\begin{equation}\label{eq-3.14.1}
\Omega^1_{JL/R_1}\r JL\ctimes_{R_2} L_2^\ast:\phi d\theta\mapsto\phi\ctimes\tilde\theta
\end{equation}
with $\tilde\theta(l)\overset{\text{def}}{=}{}^2\nabla_l(\theta)$ for any $l\in L_2$. If we respectively denote 
\eqref{eq-3.14.1} and \eqref{eq-3.3.0} by $u$ and $v$ then by definition we have 
$$
u(\xi)(l)=\xi(v(l))\quad(\forall\xi\in\Omega^1_{JL/R_1})\,.
$$
It then follows from taking the $R$-dual of \eqref{eqlem331} that \eqref{eq-3.14.1} is an isomorphism, and that the 
restriction $L_2^*\r \Omega^1_{JL/R_1}$ to $L_2^\ast$ of its inverse fits into the commutative diagram 
\begin{equation}\label{eq-arraycomp}
\xymatrix{
L_2^*\otimes L_2\ar[rrr]^{\text{contraction}}\ar[d] & & & R_2\ar[d]^{\alpha_2} \\
\Omega^1_{JL/R_1}\hat\otimes{\rm Der}_{R_1}(JL)\ar[rrr]_{\text{contraction}} & & & JL
}\,.
\end{equation} 
The next result follows by applying $\wedge(-)$ to the inverse of
\eqref{eq-3.14.1} and 
is parallel to lemma \ref{ref-3.3.1-24}: 
\begin{lemmas} The maps $R_2\r JL$
  and $L_2^*\r \Omega^1_{JL/R_1}$ induce a DG-algebra morphism $$
  \wedge_{R_2}L_2^*\r \Omega_{JL/R_1} $$
  that extends to an
  isomorphism $JL\ctimes_{R_2}\wedge_{R_2}L_2^*\r \Omega_{JL/R_1}$ of
  $JL$-modules. This isomorphism is compatible with differentials if we 
  put on $JL\ctimes_{R_2}\wedge_{R_2}L_2^*$ the canonical differential
  obtained from the $L_2$ connection on $JL$. 
\end{lemmas}
\begin{proof} We may view \eqref{eq-3.14.1} as an isomorphism between
  the Lie algebroids over $JL$ given by $JL\otimes_{R_2} L_2$ and
$\Der_{R_1}(JL)$ where the first Lie algebroid structure is deduced from
the $L_2$-connection on $JL$.

As a result we obtain an isomorphism between the corresponding DG-algebras:
\[
JL\ctimes_{R_2}\wedge_{R_2}L_2^*=\bigwedge_{JL}(JL\ctimes_{R_2}L^\ast)\cong
\bigwedge_{JL}\Der_{R_1}(JL)^\ast=
\Omega_{JL/R_1}
\]
The statement of the lemma follows easily from this. 
\end{proof}
We also have the following analogue of Lemma \ref{ref-3.3.4-29}: 
\begin{lemmas}
\label{lemma3.3.6}
Let $C$ be a commutative $R_1$-DG-algebra. The canonical map
\begin{equation}\label{eq-C-forms}
(C\ctimes_{R_1} JL) \ctimes_{R_2} (\wedge_{R_2}L_2^*) \r \Omega_{C\ctimes_{R_1} JL/C}
\end{equation}
is an isomorphism. If $JL$ carries a flat $C$-connection $\nabla$
which is compatible with the $R_2$-action on $C\ctimes_{R_1} JL$
then $\nabla\otimes \id$ on the left corresponds to the differential $d^\nabla$ on the right. 
\end{lemmas}
Here $d^\nabla$ is characterized by the properties that it coincides with $\nabla$ on 
$C\ctimes_{R_1}JL=\Omega^0_{C\ctimes_{R_1}JL/C}$ and that it commutes with the De Rham 
differential. 

\section{Coordinate spaces}\label{ref-4-31}

We keep the notations of the previous section. 

\subsection{The coordinate space of a Lie algebroid} The following
definition is inspired by~\cite{ye3}. 
\label{seccoord}
We define $R^{\coord,L}$ as the commutative $R_1$-algebra which trivializes $JL$, 
i.e.\ which is universal for the property that there is an isomorphism of 
$R^{\coord,L}$-algebras
\begin{equation}\label{ref-4.1-32}
t:R^{\coord,L}\ctimes_{R_1} JL\cong R^{\coord,L}[[t_1,\ldots,t_d]]
\end{equation}
such that $R^{\coord,L}\ctimes_{R_1} J^cL$ is mapped to $(t_1,\ldots,t_d)$. 

For use below we give an explicit description of $R^{\coord,L}$.
Since we have assumed $L$ to be free of rank $d$ we may assume that
\[
JL\cong R_1[[x_1,\ldots,x_d]]
\]
where $(x_i)_i$ is mapped to a basis of $J^cL/(J^cL)^2\cong L^\ast$.
Let $T$ be the polynomial ring over $R_1$ in the variables
$y_{i,a_1\cdots a_d}$ where $i=1,\ldots,d$, $a_j\in \NN$ and
$(a_1,\ldots,a_d)\neq (0,\ldots,0)$.  Then $R^{\coord,L}$ is the
localization of $T$ at $\det(y_{i,e_j})$ where
$e_j=(0,\ldots,1,\ldots,0)$ with the $1$ occurring in the $j$'th place
and $t$ is given by
\begin{equation}
\label{ref-4.2-33}
t(x_i)=\sum_{i,a_1\cdots a_d} y_{i,a_1\cdots a_d}t_1^{a_1}\cdots t_d^{a_d}
\end{equation}
Since $R^{\coord,L}$ is universal any $k$-linear automorphism $\alpha$ of 
$k[[t_1,\ldots,t_d]]$ yields a corresponding \emph{unique} 
$R_1$-linear automorphism $\bar{\alpha}$ of $R^{\coord,L}$ such that the 
combined automorphism $(\bar{\alpha},\alpha)$ of 
$R^{\coord,L}[[t_1,\ldots,t_d]]$ leaves the image under $t$ of $JL$ pointwise 
invariant.  Since $\Gl_d(k)$ acts on $k[[t_1,\ldots,t_d]]$ we obtain 
a corresponding $R_1$-linear action of $\Gl_d(k)$ on $R^{\coord,L}$.

In fact if we  write 
\begin{equation}\label{ref-4.3-34}
R^{\coord,L}\cong R_1\otimes S^{\coord}
\end{equation}
with $S^{\coord}=k[(y_{i,a_1,\ldots,a_d})]_{\det(y_{i,e_j})}$
then the $\Gl_d(k)$ action on $R^{\coord,L}$ is obtained from
a $\Gl_d(k)$-action on $S^{\coord}$ which  preserves
the finite dimensional vector spaces with basis $\{y_{i,a_1,\ldots,a_d}
\mid \sum_i a_i=s\}$. Needless to say that the decomposition \eqref{ref-4.3-34} 
depends on the choice of the generators $(x_i)_i$ of $J^cL$. 

It follows in particular that the $\Gl_d(k)$-action is rational. Hence
we may consider the derived action of $\mathfrak{gl}_d(k)$ on
$R^{\coord,L}$.
The action of $\Gl_d(k)$ on $k[[t_1,\ldots,t_d]]$ also yields a derived action
of $\mathfrak{gl}_d(k)$. The two actions
of $\mathfrak{gl}_d(k)$ satisfy the following compatibility.
\begin{lemmas}\label{ref-4.1.1-35}
For $v\in \frak{gl}_d(k)$ let $L_v$ be the action $v$ on $R^{\coord,L}[[t_1,\ldots,t_d]]$ 
obtained by linearly extending the action of $v$ on $k[[t_1,\ldots,t_d]]$.
 
Let $L_{\bar{v}}$ be the action of $v$ on $R^{\coord,L}[[t_1,\ldots,t_d]]$ 
obtained by linearly extending the action of $v$ on $R^{\coord,L}$.  
Then $L_{\bar{v}}+L_v$ is zero on $t(JL)$. 
\end{lemmas}
\begin{proof} This is the derived version of the fact that 
$\Gl_d(k)$ leaves $t(JL)$ pointwise fixed. 
\end{proof}

\subsection{On some DG-algebras associated to coordinate spaces}

\subsubsection{The DG-algebra $C^{\coord,L}$}\label{ref-4.2.1-36}

Put $C^{\coord,L}=\Omega_{R^{\coord,L}} \otimes_{\Omega_{R_1}} \wedge_{R_1}
L_1^\ast$. Using the DG-algebra structures on $\wedge_{R_1} L_1^\ast$ we see that 
$C^{\coord,L}$ is naturally a commutative $\wedge_{R_1} L_1^\ast$-DG-algebra. 

As we have not put any restrictions on $R$, the DG-algebras
$\Omega_{R_1}$ and $\Omega_{R^{\coord,L}}$ may be enormous objects.
However only their ``difference'' matters and this is controlled by
\eqref{ref-4.3-34}.  In fact from \eqref{ref-4.3-34} we obtain a
coordinate dependent isomorphism of DG-algebras
\[
C^{\coord,L}=\Omega_{S^{\coord}}\otimes\wedge_{R_1}L_1^\ast
\]

We will now form the completed tensor product  over
$R^{\coord,L}$ of the domain and codomain of the map \eqref{ref-4.1-32} with 
$C^{\coord,L}$ ignoring the differentials. For the domain we get
\begin{equation}\label{ref-4.4-37}
C^{\coord,L} \ctimes_{R^{\coord,L}}R^{\coord,L}\ctimes_{R_1} JL
=C^{\coord,L}\ctimes_{R_1} JL
=\Omega_{R^{\coord,L}}\ctimes_{\Omega_{R_1}} (\wedge_{R_1} L_1^\ast
\otimes_{R_1} JL)
\end{equation}
For the codomain we get
\[
C^{\coord,L}\ctimes_{R^{\coord,L}}R^{\coord,L}[[t_1,\ldots,t_d]]=
C^{\coord,L}[[t_1,\ldots,t_d]]
\]
So we obtain an isomorphism of graded algebras
\[
\tilde{t}:C^{\coord,L}\ctimes_{R_1}JL\longrightarrow C^{\coord,L}[[t_1,\ldots,t_d]]
\]
Both domain and codomain of $\tilde{t}$ carry a natural differential.
The differential on $C^{\coord,L}\ctimes_{R_1}JL$ is obtained by
combining the ordinary differential on $\Omega_{R^{\coord,L}}$ and the
differential ${}^1\nabla$ on $\wedge_{R_1}L^\ast_1\otimes_{R} JL$ using the right hand side of \eqref{ref-4.4-37}. 
The differential on $C^{\coord,L}[[t_1,\ldots,t_d]]$ is obtained from
extending the differential on $C^{\coord,L}$. Let us denote the
resulting differentials by ${}^1\nabla^{\coord}$ and $d$ respectively. 

The constructed differentials transform the obvious morphisms of graded algebras
\begin{align*}
  C^{\coord,L}&\r \Omega_{R^{\coord,L}}\ctimes_{\Omega_R} (\wedge_{R_1} L_1^\ast\otimes_{R_1} JL)\\
  C^{\coord,L} &\r C^{\coord,L}[[t_1,\ldots,t_d]]
\end{align*}
into morphisms of DG-algebras. 

 By
the middle equality in \eqref{ref-4.4-37} ${}^1\nabla^{\coord}$ may be
viewed as a flat $C^{\coord,L}$-connection on $JL$.
The map
\begin{equation}
\label{eqmc}
(\tilde{t}\circ {}^1\nabla^{\coord} \circ \tilde{t}^{-1}-d):
C^{\coord,L}[[t_1,\ldots,t_d]]\r C^{\coord,L}[[t_1,\ldots,t_d]]
\end{equation}
is now a $C^{\coord,L}$-linear derivation. 

From \cite[\S 6.4]{VdB35} we obtain the existence of  elements  $\omega^i\in C^{\coord,L}[[t_1,\ldots,t_d]]_1$ such that for 
\begin{equation}
\label{ref-4.5-38}
\omega=\sum_i \omega^i \frac{\partial\ }{\partial t_i}\in 
C^{\coord,L}\ctimes \Der_k(k[[t_1,\ldots,t_d]])
\end{equation}
we have $\tilde{t}\circ {}^1\nabla^{\coord} \circ \tilde{t}^{-1}=d+\omega$ 
and furthermore $\omega$ satisfies the Maurer-Cartan equation 
\[
d\omega+\frac{1}{2}[\omega,\omega]=0
\]
in $C^{\coord,L}\ctimes \Der_k(k[[t_1,\ldots,t_d]])$.

\subsubsection{The $\mathfrak{gl}_d(k)$-action on $C^{\coord,L}$ and the Maurer-Cartan form}
We need the following result. 
\begin{lemmas}\label{ref-4.2.1-39}
For $v\in \frak{gl}_d(k)$ let $i_{\bar{v}}$ be the derivation
on $C^{\coord,L}=\Omega_{R^{\coord,L}}\otimes_{\Omega_{R_1}} \wedge_{R_1} L_1^\ast$ 
obtained by linearly extending the contraction
of $\Omega_{R^{\coord,L}}$ with  the
action as  $R_1$-derivation of $v$ on $R^{{\coord,L}}$ (cfr \S\ref{seccoord}). 
Extend $i_{\bar{v}}$ to a map of degree $-1$ 
from $C^{\coord,L}\ctimes \Der_k(k[[t_1,\ldots,t_d]])$ to itself. Then we have
\begin{equation}\label{ref-4.6-40}
i_{\bar{v}}\omega=1\otimes v
\end{equation}
where both sides are considered as elements of $R^{\coord,L}\ctimes \Der_k(k[[t_1,\ldots,t_d]])$.
\end{lemmas}
\begin{proof}
We may prove \eqref{ref-4.6-40} by evaluation on an arbitrary element $g\in R^{\coord,L}[[t_1,\ldots,t_d]]$. 
Since $\omega=\sum_i \omega_i (\partial/\partial t_i)$ we have 
$i_{\bar{v}}(\omega)=\sum_i i_{\bar{v}}(\omega_i)(\partial/\partial t_i)$ and hence 
$(i_{\bar{v}}\omega)(g)=i_{\bar{v}}(\omega(g))$. Thus we need to show
\begin{equation}\label{ref-4.7-41}
i_{\bar{v}}\circ \omega=L_v
\end{equation}
as operators on $R^{\coord,L}[[t_1,\ldots,t_d]]$ (as in Lemma \ref{ref-4.1.1-35}
$L_v$ is the extension of the $v$-action on $k[[t_1,\ldots,t_d]]$). 

It is clear that $R^{\coord,L}[[t_1,\ldots,t_d]]$ is topologically generated by $R^{\coord,L}$ and $t(JL)$. 
Since the operators occuring are $R^{\coord,L}$-linear it is sufficient to prove the identity
\begin{equation}\label{ref-4.8-42}
i_{\bar{v}}\circ \omega\circ t=L_v\circ t
\end{equation}
as operators on $JL$. We may rewrite the l.h.s. of \eqref{ref-4.8-42} as follows
\begin{equation}
  i_{\bar{v}}\circ \omega\circ t=i_{\bar{v}}\circ (d+\omega)\circ
  t-i_{\bar{v}}\circ d\circ t =i_{\bar{v}}\circ \tilde{t} \circ
  {}^1\nabla^{\coord}-L_{\bar{v}}\circ t
  =L_{v}\circ t\,.
\end{equation}
In the second equality we have used the Cartan relation
$L_{\bar{v}}=d\circ i_{\bar{v}}+i_{\bar{v}}\circ d$ and the fact that
the term $d\circ i_{\bar{v}}$ acts as zero on
$R^{\coord,L}[[t_1,\ldots,t_d]]$ (for degree reasons). We have also
used \eqref{eqmc}. 

In the third equality we have used the fact that $JL$ is mapped to $\wedge_{R_1} L_1^\ast\otimes_{R_1}
JL$ under ${}^1\,\nabla^{\coord}$ 
and the image of $\wedge_{R_1} L_1^\ast\otimes_{R_1}
JL$ in $C^{\coord,L}\ctimes \Der_k(k[[t_1,\ldots,t_d]])$ under $t$ lies
in the part generated by $R^{\coord,L}$ and $\Der_k(k[[t_1,\ldots,t_d]])$
and on this part $i_{\bar{v}}$ is zero. 
\end{proof}

\subsection{The affine coordinate space of a Lie algebroid}
\label{ref-4.3-43}

We put $R^{\aff,L}=(R^{\coord,L})^{\Gl_d(k)}$. We easily verify from
\eqref{ref-4.3-34} that $R^{\aff,L}$ is of the form $R_1\otimes
S^{\aff}$ with $S^{\aff}= (S^{\coord})^{\Gl_d(k)}$  and furthermore
$S=k[(z_i)_i]$ for a set of (infinitely many) variables
$(z_i)_i$.

Below we will also use  the DG-algebra
$
C^{\aff,L}=\Omega_{R^{\aff,L}}\ctimes_{\Omega_{R_1}}\wedge_{R_1}L_1^\ast
$.
Exactly as for $C^{\coord,L}$ one produces a flat $C^{\aff,R}$-connection on $JL$ 
which we denote by ${}^1\nabla^{\aff}$. Furthermore we have a coordinate
dependent isomorphism of DG-algebras
\begin{equation}\label{ref-4.9-44}
C^{\aff,L}=\Omega_{S^{\aff}}\otimes \wedge_{R_1} L_1^\ast\,.
\end{equation}

\begin{lemmas}\label{ref-4.3.1-45}
For any free $R_2$-module $M$ we have that 
\[
M\r (C^{\aff,L}\ctimes_{R_1} JL,{}^1\nabla^{\aff})\ctimes_{R_2} M
\]
is a quasi-isomorphism. 
\end{lemmas}
\begin{proof}
Using the decomposition \eqref{ref-4.9-44} we need to show that 
\begin{equation}\label{ref-4.10-46}
M\r \Omega_{S^{\aff}}\otimes \wedge_{R_1} L_1^\ast\ctimes_{R_1} JL\ctimes_{R_2}M
\end{equation}
is a quasi-isomorphism. We filter the r.h.s. of \eqref{ref-4.10-46} with the $J^cL$-adic filtration. 
This means we have to show that
\[
M \r \Omega_{S^{\aff}}\otimes\wedge_RL^\ast\otimes_R SL^\ast\otimes_R M
\]
is a quasi-isomorphism.

Using the proof of Proposition \ref{ref-3.2.3-21} we see that we 
may replace $\wedge_RL^\ast\otimes_R S_RL^\ast$ by $R$. Furthermore since $S^{\aff}$
is a polynomial ring we also find that $\Omega_{S^{\aff}}$ is quasi-isomorphic
to $k$. Thus we are done. 
\end{proof}

\subsection{The universal property of the affine coordinate space}

The affine coordinate space has a universal property, similar to  \eqref{ref-4.1-32}.
\begin{propositions} \label{ref-4.4.1-47} There is a filtered isomorphism of $R$-algebras
\begin{equation}\label{ref-4.11-48}
R^{\aff,L}\ctimes_R JL\cong R^{\aff,L}\ctimes_R \widehat{SL^\ast}
\end{equation}
which induces the canonical isomorphism
\[
R^{\aff,L}\otimes_R \gr(JL)\cong R^{\aff,L}\otimes_R SL^\ast
\]
obtained by extending \eqref{ref-3.8-20}. Moreover $R^{\aff,L}$ is
universal for the existence of such an isomorphism.
\end{propositions}
\begin{proof} We start with the isomorphism
\begin{equation}
\label{ref-4.12-49}
R^{\coord,L}\ctimes_R JL\cong R^{\coord,L}[[t_1,\ldots,t_d]]
\end{equation}
It follows from the discussion after \eqref{ref-4.2-33} that this
isomorphism if $\Gl_d$-equivariant is we equip the righthand side with
a $\Gl_d$-action which is a combination of the linear action on the
$(t_i)_i$'s and the extension of the $\Gl_d$-action on $R^{\coord,L}$. In
particular this $\Gl_d$-action preserves the $(t_i)_i$-grading.

Put
\[
\tilde{L}^\ast=R^{\coord,L} t_1+\cdots+R^{\coord,L} t_d
\]
considered as a $\Gl_d$-module.
Then (tautologically) we have a $\Gl_d$-equivariant isomorphism
\[
R^{\coord,L}[[t_1,\ldots,t_d]]\cong S_{R^{\coord,L}}(\tilde{L}^\ast)\,\hat{}
\]
Combining this with \eqref{ref-4.12-49} we get a $\Gl_d$-equivariant
isomorphism
\[
R^{\coord,L}\ctimes_R JL\cong S_{R^{\coord,L}}(\tilde{L}^\ast)\,\hat{}
\]
and looking at degree one of the associated graded rings:
\[
R^{\coord,L}\ctimes_R L^\ast\cong \tilde{L}^\ast
\]
Thus 
\begin{equation}
\label{ref-4.13-50}
R^{\coord,L}\ctimes_R JL\cong S_{R^{\coord,L}}(R^{\coord,L}\ctimes_R L^\ast)\,\hat{}\cong
R^{\coord,L}\ctimes\widehat{SL^\ast}
\end{equation}
It now suffices to take $\Gl_d$-invariants to get the isomorphism \eqref{ref-4.11-48}.

\medskip

We will only sketch the proof of universality since we will not need it
in the sequel. Assume that we have an isomorphism
\[
W\ctimes_R JL\cong W\ctimes_R \widehat{SL^\ast}
\]
such that
\[
W\otimes_R \gr(JL)\cong W\otimes_R SL^\ast
\]
We need to construct a corresponding morphism $R^{\aff,L}\r W$. 

We let $\tilde{W}$ be the commutative $W$-algebra which is universal for
the existence of an isomorphism
\[
\tilde{W}\otimes_W (W\otimes_R L)\cong \tilde{W}t_1+\cdots+\tilde{W}t_d
\]
Thus $\Spec \tilde{W}/\Spec W$ is a $\Gl_d$-torsor and in particular
$\tilde{W}^{\Gl_d}\cong W$. We then have
\begin{align*}
\tilde{W}\ctimes_R JL&=\tilde{W}\ctimes_{W} (W\ctimes_R JL)\\
&=\tilde{W}\ctimes_{W}(W\ctimes_R \widehat{SL^\ast})\\
&=\tilde{W}\ctimes_R \widehat{SL^\ast}\\
&=S_{\tilde{W}}(\tilde{W}t_1+\cdots+\tilde{W}t_d)\,\hat{}\\
&=\tilde{W}[[t_1,\ldots,t_d]]
\end{align*}
Hence there exists a corresponding morphism $R^{\coord,L} \r\tilde{W}$. 
Taking $\Gl_d$-invariants yields the requested morphism $R^{\aff,L}\r W$.
One easily checks that this morphism satisfies the appropriate uniqueness
properties. 
\end{proof}
\section{$L_\infty$-algebras}\label{ref-5-51}
In this section we recall some properties of
$L_\infty$-algebras and we fix some notations. 
\subsection{$L_\infty$-algebras and morphisms}
An {\em $L_\infty$-structure} on a vector space $\frak{g}$ is a coderivation $Q$ of
degree one on $S(\frak{g}[1])$ which has square zero. Such a coderivation
is fully determined its ``Taylor coefficients'' which are the coefficients
\[
Q_i:S^i(\frak{g}[1])\xrightarrow{\text{inclusion}}
S(\frak{g}[1])\xrightarrow{Q}
S(\frak{g}[1])\xrightarrow{\text{projection}} \frak{g}[1]
\]
If for $a,b\in \frak{g}$ one puts
\begin{equation}\label{ref-5.1-52}
da=-Q_1(a)\quad\textrm{and}\quad[a,b]=(-1)^{|a|} Q_2(a,b)\,.
\end{equation}
then $d^2=0$ 
and  $d$ is a derivation of degree one of $\frak{g}$ with respect to the binary operation of degree 
zero $[-,-]$. If $\partial^i Q=0$ for $i>2$ then $\frak{g}$ is a DG-Lie algebra. Conversely any DG-Lie
algebra can be made into an $L_\infty$-algebra by defining $Q_1$, $Q_2$ according to 
\eqref{ref-5.1-52} and by putting $Q_i=0$ for $i>2$.

\medskip

A morphism of $L_\infty$-algebras $\frak{g}\r \frak{h}$, or {\em $L_\infty$-morphism} is by definition
an augmented coalgebra map of degree zero $S(\frak{g}[1])\r S(\frak{h}[1])$ 
commuting with $Q$. A morphism of $L_\infty$-algebras is again determined
by its Taylor coefficients
\[
\psi_i:S^i(\frak{g}[1])\xrightarrow{\text{inclusion}}
S(\frak{g}[1])\xrightarrow{\psi}
S(\frak{h}[1])\xrightarrow{\text{projection}} \frak{h}[1]
\]
One has $d\psi_1=\psi_1 d$
and hence $\psi_1$ defines a morphism of complexes. 

\medskip

The above notions make sense in any symmetric monoidal category. We
will use them in the case of filtered complete linear topological
vector spaces. Of course in that case the symmetric products have to
be replaced by completed symmetric products. 

\subsection{Twisting}\label{ref-5.2-53}

Assume that $\psi:\frak{g}\r \frak{h}$ is a $L_\infty$-morphism
between $L_\infty$-algebras. We assume in addition that we are in a
complete filtered setting (in the category of graded vector spaces).
I.e.\ there are ascending filtrations $(F_n \frak{g})_n$
$(F_n\frak{h})_n$, on $\frak{g}$, $\frak{h}$ and furthermore
$\frak{g}$, $\frak{h}$ are graded complete for the topologies induced
by these filtrations. We assume in addition that $\psi$ is compatible with
the filtrations (i.e.\ it is a filtered map for the induced
filtrations on $S(\frak{g}[1])$ and $S(\frak{h}[1])$).

Let $\omega\in F_{-1}\frak{g}_1$ be a solution of the
$L_\infty$-Maurer-Cartan equation in $\frak{g}$.
\begin{equation}\label{ref-5.2-54}
\sum_{i\ge 1}\frac{1}{i!} Q_i(\omega^i)=0\,.
\end{equation}
Define $Q_\omega$, $\psi_\omega$ and $\omega'$ by 
\begin{align}
\label{ref-5.3-55}
Q_{\omega,i}(\gamma)&=\sum_{j\ge 0} \frac{1}{j!} Q_{i+j}
(\omega^j \gamma)\qquad \text{(for $i>0$)}\\
\label{ref-5.4-56}
\psi_{\omega,i}(\gamma)&=\sum_{j\ge 0} \frac{1}{j!} \psi_{i+j}
(\omega^j \gamma)\qquad \text{(for $i>0$)}\\
\label{ref-5.5-57}
\omega'&=\sum_{j\ge 1} \frac{1}{j!} \psi_j
(\omega^j )
\end{align}
for $\gamma\in S^i(\frak{g}[1])$. Then by  \cite[Thm 3.21,3.27]{ye7} 
$\omega'\in F_{-1}\frak{h}$ is a solution of the Maurer-Cartan
equation in $\frak{h}$ and furthermore $\frak{g}$, $\frak{h}$, when
equipped with $Q_\omega$, $Q_{\omega'}$ are again $L_\infty$-algebras.
If we denote these by $\frak{g}_\omega$ and $\frak{h}_{\omega'}$ then
$\psi_\omega$ defines a filtered $L_\infty$-map $\frak{g}_\omega\r
\frak{h}_{\omega'}$. 
\begin{remarks} Formulas similar to
  (\ref{ref-5.3-55}-\ref{ref-5.5-57}) also appear at other places in
  the literature e.g. \cite[eq. (60)(61)]{Dogushev} and \cite[eq.
  (4.4)]{Tsygan}. They are implicit in the language of formal
  $Q$-manifolds employed by Kontsevich in \cite{Ko3},
\end{remarks}
If $\frak{g}$ is a DG-Lie algebra then the $L_\infty$-Maurer-Cartan equation 
translates into the usual Maurer-Cartan equation
\begin{equation}\label{ref-5.6-58}
d\omega+\frac{1}{2}[\omega,\omega]=0
\end{equation}
and we obtain $Q_{\omega,1}(\gamma)=Q_1(\gamma)+Q_2(\omega\gamma)$, 
$Q_{\omega,2}(\gamma)=Q_2(\gamma)$ and $Q_{\omega,i}(\gamma)=0$ for $i\ge 3$. 
Translated into differentials and Lie brackets we get 
\begin{equation}\label{ref-5.7-59}
d_\omega=d+[\omega ,-]\quad\textrm{and}\quad[-,-]_\omega=[-,-]\,.
\end{equation}
\subsection{Descent for $L_\infty$-morphisms}\label{ref-5.3-60}
Assume that $\frak{g}$ is an algebra over a DG-operad~$\Oscr$ with underlying
graded operad $\widetilde{O}$ and 
consider a set of $\widetilde{O}$-derivations of degree -1 $(i_v)_{v\in S}$ on
$\frak{g}$. Put $L_v=di_v+i_vd$.
This is a derivation of $\frak{g}$ of degree zero which commutes with $d$. 

Put 
\begin{equation}\label{ref-5.8-61}
\frak{g}^{S}=\{w\in \frak{g}\mid \forall v\in S:i_v w= L_v w=0\}
\end{equation}
It is easy to see that $\frak{g}^{S}$ is an algebra over $\Oscr$ as well. Informally
we will call such a set of derivations $(i_v)_{v\in S}$ an $S$-action.
\begin{remarks}
By definition the
notion of an $S$-action only depends on the graded structure of~$\frak{g}$. However the construction of $\frak{g}^{S}$ also depends
on the differential. 
\end{remarks}
The following is a slightly strengthened version of \cite[Prop.\ 7.6.3]{VdB35}.
\begin{propositions}\label{ref-5.3.2-62}
  Assume that $\psi$ is an $L_\infty$-morphism $\frak{g}\r \frak{h}$
  between $L_\infty$-algebras equipped with an $S$-action as
  above. Assume that $\psi$ commutes with the $S$-action in the sense
  that for all $v\in S$,
\[
i_v\psi_n(w_1,\ldots,w_n)=\sum (-1)^{|w_1|+\cdots+|w_{l-1}|-(l-1)} \psi_n(w_1,\ldots,
i_v(w_l),\ldots,w_n)
\]
Then $\psi$ descends to an
$L_\infty$-morphism $\psi^{S}:\frak{g}^{S}\r\frak{h}^{S}$. 
\end{propositions}

\subsection{Compatibility with twisting}

Assume that $\frak{g}$, $\frak{h}$ are topological $DG$-Lie algebras
and $\psi$ is an $L_\infty$-morphism $\frak{g}\r \frak{h}$. As in
\S\ref{ref-5.2-53} we assume that $\frak{g}$, $\frak{h}$ are complete
filtered and $\psi$ is compatible with the filtration. Our aim is to
understand the behavior of $S$-actions under twisting.

Assume that $\frak{g}$ and $\frak{h}$ are equipped with a $S$-action and assume that $\psi$ commutes 
with this action (as in Proposition \ref{ref-5.3.2-62}). Let $\omega\in F_{-1}\frak{g}_1$ be a solution 
to the Maurer-Cartan equation. 
Since twisting does not change the Lie bracket (see \eqref{ref-5.7-59}), 
$S$ acts on  $\frak{g}_\omega$ and $\frak{h}_\omega$ as well. The following
is \cite[Prop.\ 7.7.1]{VdB35}.
\begin{propositions}\label{ref-5.4.1-63}
Assume that for $i\ge 2$ and all $v\in S$, $\gamma\in S^{i-1}(\frak{g}[1])$ we have
\begin{equation}
\label{ref-5.9-64}
\psi_i(i_v\omega \cdot \gamma)=0
\end{equation}
Then $\psi_\omega$ commutes with the $S$-action on  $\frak{g}_\omega$ and $\frak{h}_{\omega'}$.
\end{propositions}
\section{Formality for Lie algebroids}
In this section we prove Theorem \ref{ref-1.6-8}.  We first prove a more
precise result in the ring case. To do so we use the existence of the
desired $L_\infty$-quasi-isomorphism in the local case, and extend it
to the ring case with the help of coordinate spaces constructed in the
previous section.  We end the proof by sheafifying the ring case,
using appropriate functorial properties.
\subsection{The formality in the ring case and its functorial properties}

\begin{theorem}\label{ref-6.1-65}
  Let $R$ be a $k$-algebra. Assume that $L$ is a Lie-algebroid over
  $R$ which is free of rank $d$.
There exists a canonical DG-Lie algebra
$\mathfrak{l}^L$ together with $L_\infty$-quasi-isomorphisms
\begin{equation}\label{ref-6.1-66}
T^{L}_{\poly}(R)\longrightarrow\mathfrak{l}^L\longleftarrow D^{L}_{\poly}(R)
\end{equation}
such that the induced map
\[
\mu:T^{L}_{\poly}(R)\longrightarrow H^\ast(D^{L}_{\poly}(R))
\]
is given by the HKR-formula \eqref{ref-3.1-12}. 

The DG-Lie algebra $\mathfrak{l}^L$ and the quasi-isomorphisms in \eqref{ref-6.1-66} are functorial
in the following sense: assume that $\phi:(R,L)\r (T,M)$ is an algebraic morphism of Lie algebroids 
which induces an isomorphism
\begin{equation}\label{ref-6.2-67}
T\otimes_R L\cong M
\end{equation}
then there is an associated commutative diagram
\begin{equation}\label{ref-6.3-68}
\xymatrix{
T^{L}_{\poly}(R)\ar[r]\ar[d] & \mathfrak{l}^L\ar[d]^{\mathfrak{l}^{\phi}} & D^{L}_{\poly}(R)\ar[l]\ar[d]\\
T^{M}_{\poly}(T)\ar[r] & \mathfrak{l}^M & D^{M}_{\poly}(T)\ar[l]
}
\end{equation}
and $\mathfrak{l}^{\phi\theta}=\mathfrak{l}^{\phi}\circ\mathfrak{l}^{\theta}$. 
\end{theorem}
The proof of this theorem will take the greater part of the next two
subsections.

\subsection{The local formality quasi-isomorphism}\label{ref-6.2-69}

Let $F=k[[t_1,\ldots,t_d]]$. Kontsevich (over the reals) and Tamarkin
(over the rationals) construct 
an $L_\infty$-quasi-isomorphism \cite{Ko3,Tamarkin} 
\begin{equation}\label{ref-6.4-70}
\Uscr:T_{\poly}(F)\r D_{\poly}(F)
\end{equation}
where $\Uscr_1$ is given by the HKR formula\footnote{The sign
$(-)^{p(p-1)/2}$ is not present in Kontsevich's setting. I this paper we
  slightly modify Kontsevich's quasi-isomorphism. See \S\ref{ref-8-109}.} 
\begin{equation}\label{ref-6.5-71}
\Uscr_1(\partial_{i_1}\wedge\cdots \wedge\partial_{i_p})= 
(-1)^{p(p-1)/2}\frac{1}{p!}\sum_{\sigma\in S_p}(-1)^{\sigma}\partial_{i_{\sigma(1)}}\otimes \cdots \otimes \partial_{i_\sigma(p)}
\end{equation}
with $\partial_i=\partial/\partial t_i$. 
This quasi-isomorphism has two supplementary 
properties which are crucial for its extension to the global case.
\begin{enumerate}
\item[(P4)]
$\Uscr_q(\gamma_1\cdots \gamma_q)=0$ for $q\ge 2$ and 
$\gamma_1,\ldots,\gamma_q\in T^{\text{poly},1}(F)$.\footnote{For degree reasons, 
this is always true if $q>2$. }
\item[(P5)]
$\Uscr_q(\gamma\alpha)=0$ for $q\ge 2$ and $\gamma\in
\frak{gl}_d(k)\subset T^{\text{poly},1}(F)$. 
\end{enumerate}
For Tamarkin's quasi-isomorphism
the fact that properties (P4) and (P5) hold has been proved in \cite{halb}. 

\subsection{Proof of Theorem \ref{ref-1.6-8} in the ring case}\label{ref-6.3-72}

\subsubsection{Resolutions}

In this section we construct resolutions of $T_{\poly}^L(R)$ and $D_{\poly}^L(R)$. 
These are jet analogues of the {\it Dolgushev-Fedosov resolutions} in \cite[Section 2]{cal}. \\

Since the action of $UL_2$ on $C^{\aff,L}\ctimes_{R_1} JL$ commutes
with ${}^1\nabla^{\aff}$ we obtain morphisms of DG-Lie algebras:
\begin{equation}\label{ref-6.6-73}
\begin{split}
  T_{\poly}^{L_2}(R_2)&\r  T_{\poly,C^{\aff,L}}(C^{\aff,L}\ctimes_{R_1} JL,{}^1\nabla^{\aff})\\
  D_{\poly}^{L_2}(R_2)&\r D_{\poly,C^{\aff,L}}(C^{\aff,L}\ctimes_{R_1}
  JL,{}^1\nabla^{\aff})\,.
\end{split}
\end{equation}
\begin{propositions}
The morphisms in \eqref{ref-6.6-73} are quasi-isomorphisms. 
\end{propositions}
\begin{proof}
This follows from lemmas \ref{ref-3.3.4-29} and \ref{ref-4.3.1-45}. 
\end{proof}

\subsubsection{The formality map on coordinate spaces}

The local $L_\infty$-quasi-isomorphism 
\[
\Uscr:T_{\poly}(F)\r D_{\poly}(F)
\]
extends linearly to an $L_\infty$-quasi-isomorphism 
\[
\tilde{\Uscr}:C^{\coord,L}\ctimes T_{\poly}(F)
\r C^{\coord,L}\ctimes  D_{\poly}(F)
\]
One easily verifies that the canonical maps
\begin{equation}\label{ref-6.7-74}
\begin{split}
C^{\coord,L}\ctimes T_{\poly}(F)&\r~
T_{\poly,C^{\coord,L}}(C^{\coord,L}\ctimes F)\\
C^{\coord,L}\ctimes D_{\poly}(F)&\r~
D_{\poly,C^{\coord,L}}(C^{\coord,L}\ctimes F)
\end{split}
\end{equation}
are isomorphisms of DG-Lie algebras. 
Thus we obtain a corresponding $L_\infty$-quasi-isomorphism
$$
\tilde{\Uscr}:T_{\poly,C^{\coord,L}}(C^{\coord,L}\ctimes F)
\r D_{\poly,C^{\coord,L}}(C^{\coord,L}\ctimes F)\,.
$$

In Section \ref{ref-4-31} ($\S$ \ref{ref-4.2.1-36}) we have constructed an isomorphism of 
$C^{\coord,L}$-DG-algebras
\[
\tilde{t}:(C^{\coord,L}\ctimes_{R_1} JL,{}^1\nabla^{\coord})
\r (C^{\coord,L}\ctimes F,d+\omega)\,.
\]
Therefore we obtain an isomorphism of Lie algebras
\begin{equation}\label{ref-6.8-75}
\tilde{t}^{-1}\circ - \circ \tilde{t}:D_{\poly,C^{\coord,L}}(C^{\coord,L}\ctimes_{R_1} JL)
\r D_{\poly,C^{\coord,L}}(C^{\coord,L}\ctimes F)\,.
\end{equation}
The Hochschild differential on the left is sent to the 
Hochschild differential on the right. The differential $[{}^1\nabla^{\coord},-]$
on the left is sent to $[d+\omega,-]$ on the right. Then it follows using \eqref{ref-5.7-59} 
from \S \ref{ref-5.2-53} that $\tilde{t}$ defines an isomorphism of DG-Lie algebras
\begin{equation}\label{ref-6.9-76}
D_{\poly,C^{\coord,L}}(C^{\coord,L}\ctimes_{R_1} JL,{}^1\nabla^{\coord})
\cong  D_{\poly,C^{\coord,L}}(C^{\coord,L}\ctimes F,d)_{\omega}\,.
\end{equation}
Similarly we have an isomorphism of DG-Lie algebras 
\begin{equation}\label{ref-6.10-77}
T_{\poly,C^{\coord,L}}(C^{\coord,L}\ctimes_{R_1} JL,{}^1\nabla^{\coord})
\cong  T_{\poly,C^{\coord,L}}(C^{\coord,L}\ctimes F,d)_{\omega}\,.
\end{equation}
We now use the grading on $T_{\poly,C^{\coord,L}}(C^{\coord,L}\ctimes F)$ and
$D_{\poly,C^{\coord,L}}(C^{\coord,L}\ctimes F)$ obtained from the
$C^{\coord,L}$-grading on $C^{\coord,L}\ctimes F$ as a filtration.
Thus
\[
F_{-n}(T_{\poly,C^{\coord,L}}(C^{\coord,L}\ctimes F)=
T_{\poly,C^{\coord,L}}(C^{\coord,L}\ctimes F)_{\ge n}
\]
and similarly for $D_{\poly}$.
Since these filtrations are finite in each degree for
the total gradings on
$T_{\poly,C^{\coord,L}}(C^{\coord,L}\ctimes F)$ and
$D_{\poly,C^{\coord,L}}(C^{\coord,L}\ctimes F)$ these
DG-Lie algebras are graded complete.
It is also clear that $\tilde{\Uscr}$ is compatible with $F$. Finally since
$\omega\in T_{\poly,C^{\coord,L}}(C^{\coord,L}\ctimes F)_1$
it follows that $\omega \in F_{-1}( T_{\poly,C^{\coord,L}}
(C^{\coord,L}\ctimes F))$. 
Thus the twisting formalism exhibited in \S\ref{ref-5.2-53}
applies and we obtain an $L_\infty$-morphism
\begin{equation}
\label{ref-6.11-78}
\tilde{\Uscr}_{\omega}:T_{\poly,C^{\coord,L}}(C^{\coord,L}\ctimes F)_{\omega}
\r D_{\poly,C^{\coord,L}}(C^{\coord,L}\ctimes F)_{\omega}
\end{equation}
since by (P4) (using the notation of \eqref{ref-5.5-57}) one has
\[
\omega'=\sum_{j\ge 1} \frac{1}{j!} \tilde{\Uscr}_j(\omega^j)=\tilde{\Uscr}_1(\omega)=\omega\,.
\]
Hence using \eqref{ref-6.9-76} and \eqref{ref-6.10-77} we have an $L_\infty$-morphism
\begin{equation}
\label{ref-6.12-79}
\Vscr^{\coord}:T_{\poly,C^{\coord,L}}(C^{\coord,L}\ctimes_{R_1} JL,{}^1\nabla^{\coord})
\r D_{\poly,C^{\coord,L}}(C^{\coord,L}\ctimes_{R_1} JL,{}^1\nabla^{\coord})\,.
\end{equation}
\subsubsection{The formality map on affine coordinate spaces}

First remark that $\tilde{\Uscr}_\omega$ descends under the $\mathfrak{gl}_d(k)$-action. Namely, 
given the facts that $\tilde{\Uscr}$ clearly commutes with the $\mathfrak{gl}_d(k)$-action (in the 
sense of Proposition \ref{ref-5.3.2-62}) and that, using \eqref{ref-4.6-40} and (P5), 
$\tilde{\Uscr}_i(i_{\bar{v}}(\omega)\cdot \gamma)=0$ for any $v\in\mathfrak{gl}_d(k)$ and $i\ge 2$; 
we can apply the criteria given by Propositions \ref{ref-5.3.2-62} and \ref{ref-5.4.1-63} and obtain an 
$L_\infty$-morphism 
{\footnotesize\begin{equation}\label{ref-6.13-80}
\Vscr^{\aff}:T_{\poly,C^{\coord,L}}(C^{\coord,L}\ctimes_{R_1} JL,{}^1\nabla^{\coord})^{\frak{gl}_d(k)}
\r D_{\poly,C^{\coord,L}}(C^{\coord,L}\ctimes_{R_1} JL,{}^1\nabla^{\coord})^{\frak{gl}_d(k)}\,.
\end{equation}}
Here the notation $(-)^{\mathfrak{gl}_d(k)}$ is used  in the sense of \eqref{ref-5.8-61} and 
$\mathfrak{gl}_d(k)$ acts by the derivation of the $\Gl_d(k)$-action on the factor $\Omega_{R^{\coord,L}}$ 
of $C^{\coord,L}=\Omega_{R^{\coord,L}} \otimes_{\Omega_{R_1}} \wedge_{R_1} L_1^\ast$. 

There are morphisms of DG-Lie algebras
\begin{equation}\label{ref-6.14-81}
\begin{split}
T_{\poly,C^{\aff,L}}(C^{\aff,L}\ctimes_{R_1} JL,{}^1\nabla^{\aff})
&\r T_{\poly,C^{\coord,L}}(C^{\coord,L}\ctimes_{R_1} JL,{}^1\nabla^{\coord})\\
D_{\poly,C^{\aff,L}}(C^{\aff,L}\ctimes_{R_1} JL,{}^1\nabla^{\aff})
&\r D_{\poly,C^{\coord,L}}(C^{\coord,L}\ctimes_{R_1} JL,{}^1\nabla^{\coord})
\end{split}
\end{equation}
obtained by extending $C^{\aff,L}$-linear poly-vector fields and
poly-differential operators to $C^{\coord,L}$-linear ones. We claim
that these maps yield isomorphisms of DG-Lie algebras
\begin{equation}\label{ref-6.15-82}
\begin{split}
T_{\poly,C^{\aff,L}}(C^{\aff,L}\ctimes_{R_1} JL,{}^1\nabla^{\aff})
&\r T_{\poly,C^{\coord,L}}(C^{\coord,L}\ctimes_{R_1} JL,{}^1\nabla^{\coord})^{\mathfrak{gl}_d(k)}\\
D_{\poly,C^{\aff,L}}(C^{\aff,L}\ctimes_{R_1} JL,{}^1\nabla^{\aff})
&\r D_{\poly,C^{\coord,L}}(C^{\coord,L}\ctimes_{R_1} JL,{}^1\nabla^{\coord})^{\mathfrak{gl}_d(k)}\,.
\end{split}
\end{equation}
Using Lemma \ref{ref-3.3.4-29} and using the fact that $T^{L_2}_{\poly}(R_2)$ and
$D^{L_2}_{\poly}(R_2)$ are free $R_2$-modules and that $JL$ is a topologically
free $R_1$-module it is sufficient to prove that
\begin{equation}\label{ref-6.16-83}
(\Omega_{R^{\coord,L}}\otimes_{\Omega_{R_1}}\wedge_{R_1} L_1^\ast)^{\mathfrak{gl}_d(k)}=
\Omega_{R^{\aff,L}}\otimes_{\Omega_{R_1}}\wedge_{R_1} L_{1}^\ast
\end{equation}
Using the notations of Section \ref{ref-4-31} the isomorphism \eqref{ref-6.16-83} follows from
\[
\Omega_{S^{\coord}}^{\mathfrak{gl}_d(k)}=\Omega_{S^{\aff}}
\]
This follows easily from the fact that $\Gl_d(k)$ acts freely on
$\Spec S^{\coord}$. 

Therefore \eqref{ref-6.13-80} now yields an $L_\infty$-morphism
\[
\Vscr^{\aff}:T_{\poly,C^{\aff,L}}(C^{\aff,L}\ctimes_{R_1} JL,{}^1\nabla^{\aff})
\longrightarrow D_{\poly,C^{\aff,L}}(C^{\aff,L}\ctimes_{R_1} JL,{}^1\nabla^{\aff})\,.
\]

\subsubsection{End of the proof}

We have constructed $L_\infty$-morphisms 
\begin{equation}
\label{ref-6.17-84}
\xymatrix{
T_{\poly}^{L_2}(R_2)\ar[r]^-\cong &
T_{\poly,C^{\aff,L}}(C^{\aff,L}\ctimes_{R_1} JL)
\ar[d]^{\Vscr^{\aff}}\\
D_{\poly}^{L_2}(R_2)\ar[r]_-\cong &
D_{\poly,C^{\aff,L}}(C^{\aff,L}\ctimes_{R_1} JL)\,.}
\end{equation}
such that the horizontal maps are quasi-isomorphisms.

We put $\mathfrak{l}^L=D_{\poly,C^{\aff,L}}(C^{\aff,L}\ctimes_{R_1}
JL)$. Then the lower horizontal map in \eqref{ref-6.17-84} yields the
rightmost quasi-isomorphism in 
\eqref{ref-6.1-66}.

 We will
prove below that the composition
\[
T_{\poly}^{L_2}(R_2)\r H^\ast(\mathfrak{l}^L) \xrightarrow{\cong^{-1}} H^\ast(
D_{\poly}^{L_2}(R_2))
\]
coincides with the HKR-isomorphism. It follows in particular that the
diagonal map in \eqref{ref-6.17-84} 
\[
T_{\poly}^{L_2}(R_2)\r \mathfrak{l}^L
\]
is an $L_\infty$-quasi-isomorphism as well.  This is the leftmost
quasi-isomorphim in \eqref{ref-6.1-66}.
We leave to the reader the tedious but straightforward verification
of the functoriality of $\mathfrak{l}^L$. 

To prove that the map on cohomology is given by the HKR-map we regard
the complexes occurring in \eqref{ref-6.14-81} as double complexes
such that the differential obtained from  $C^{\aff,L}$ 
and $C^{\coord,L}$ is horizontal.  We write the
coordinates for the double grading as couples $(p,q)$ where $p$ is the
column index. 

According to \eqref{ref-5.4-56}~$\tilde{\Uscr}_{\omega,1}$ is given by
\begin{equation}\label{ref-6.18-85}
\tilde{\Uscr}_{\omega,1}(\gamma)=\sum_{j\ge 0}\frac{1}{j!}
\tilde{\Uscr}_{j+1}(\omega^j\gamma)\,.
\end{equation}
Now $\tilde{\Uscr}_{j+1}$ is homogeneous for the column grading and of
degree $1-(j+1)$ for the Hochschild grading (the row grading), thus it
has bidegree $(0,-j)$. Since $\omega$ lives in $C^{\coord,L}_1\ctimes
T^0_{\poly}(F)$ it has bidegree $(1,0)$, and
hence $\tilde{\Uscr}_{j+1}(\omega^j$-$)$ has bidegree
$(j,-j)$.

Let $\tilde{\Uscr}_{\omega,1}^{j}$ be the component of 
$\tilde{\Uscr}_{\omega,1}$
indexed by $j$ in \eqref{ref-6.18-85}. 
\begin{lemmas}
We have the following commutative diagram
\begin{equation}\label{ref-6.19-86}
\begin{CD}
  T_{\poly}^{L_2}(R_2) @>>> C^{\coord,L}\ctimes T_{\poly}(F)\\
  @V \mu VV @VV \tilde{\Uscr}_{\omega,1}^{0}V\\
  D_{\poly}^{L_2}(R_2)@>>>
  C^{\coord,L}\ctimes D_{\poly}(F)
\end{CD}
\end{equation}
where the horizontal arrows are inclusions obtained from the action by derivations of $L_2$ on 
$C^{\coord,L}\ctimes_{R_1}JL\cong C^{\coord,L}\ctimes F$ (see \eqref{ref-4.1-32})
and $\mu$ is the HKR-map \eqref{ref-3.1-12}. 
\end{lemmas}
\begin{proof}
  This is almost a tautology.  Let $l_1,\dots,l_n\in L$ and denote by
  $\delta_i$ the derivation on  $C^{\coord,L}\ctimes F$
corresponding to $l_i$. Then $$
  \tilde{\Uscr}_{\omega,1}^{0}(\delta_1\wedge\dots\wedge\delta_n)
  =\tilde{\Uscr}_1(\delta_1\wedge\dots\wedge\delta_n)
  =(-1)^{n(n-1)/2}\frac1{n!}\sum_{\sigma\in
    S_{n}}\epsilon(\sigma)\delta_{\sigma(1)}\otimes\cdots\otimes\delta_{\sigma(n)}\,.
  $$
This implies the commutativity of \eqref{ref-6.19-86}.
\end{proof}

Since the maps in \eqref{ref-6.14-81} are inclusions
$\Vscr^{\aff}_1$ has the same grading properties as
$\tilde{\Uscr}_{\omega,1}$. In particular it maps
$T_{\poly,C^{\aff,L}}(C^{\aff,L}\ctimes_{R_1}JL)_{p,q}$ to $\oplus_j
D_{\poly,C^{\aff,L}}(C^{\aff,L}\ctimes_{R_1} JL)_{p+j,q-j}$.  Let
$\Vscr^{\aff,j}_1$ be the component of
$\Vscr^{\aff}_j$ indexed by $j$ in this decomposition. Thus we
obtain a commutative diagram
\begin{equation}\label{ref-6.20-87}
\begin{CD}
  T_{\poly}^{L_2}(R_2) @>>>  T_{\poly,C^{\aff,L}}
(C^{\aff,L}\ctimes_{R_1} JL)\\
  @V \mu VV @VV\Vscr^{\aff,0}_1 V\\
  D_{\poly}^{L_2}(R_2)@>>>
 D_{\poly,C^{\aff,L}}
(C^{\aff,L}\ctimes_{R_1} JL)
\end{CD}
\end{equation}
The following lemma ends the proof of the theorem (see \cite[Thm. 7.1]{ye3}). 
\begin{lemmas}
$\Vscr^{\aff,0}_1$ and $\Vscr^{\aff}_1$ induce the
same maps on cohomology. 
\end{lemmas}
\begin{proof}
We filter $T_{\poly,C^{\aff,L}}(C^{\aff,L}\ctimes_{R_1} JL)$ and 
$D_{\poly,C^{\aff,L}}(C^{\aff,L}\ctimes_{R_1} JL)$ according to the column index.

The $E_1$ term of the resulting spectral sequences consists of the
cohomology of the colums.  
Using \eqref{ref-3.13-30} we have to compute
the cohomology of $(C^{\aff,L}\ctimes_{R_1}
JL)\ctimes_{R_2}T_{\poly}^{L_2}(R_2)$ and $(C^{\aff,L}\ctimes_{R_1}
JL)\ctimes_{R_2} D_{\poly}^{L_2}(R_2)$ 
for the second factor.  We obtain $(C^{\aff,L}\ctimes_{R_1}
JL)\ctimes_{R_2}T_{\poly}^{L_2}(R_2)$ and $(C^{\aff,L}\ctimes_{R_1}
JL)\ctimes_{R_2} H^\cdot(D_{\poly}^{L_2}(R_2))$ (the latter because
$D_{\poly}^{L_2}(R_2)$ is a complex consisting of filtered projective 
$R_2$-modules with filtered projective cohomology).

Using Lemma
\ref{ref-4.3.1-45} we obtain that the $E_2$ terms are given by
$T_{\poly}^{L_2}(R_2)$ and $H^\cdot(D_{\poly}^{L_2}(R_2))$. It is now
clear that $\Vscr^{\aff,0}_1$ and $\Vscr^{\aff}_1$ induce indeed the
same map on cohomology.
\end{proof}
\subsection{Proof of  Theorem \ref{ref-1.6-8} in the sheaf case}\label{ref-6.4-88}
As indicated in the the introduction we can prove a result which slightly more
general than  Theorem \ref{ref-1.6-8}. We work over a ringed site  $(\Cscr,\Oscr)$
and $\Lscr$ is a Lie algebroid locally free of rank $d$ on $(\Cscr,\Oscr)$. 
The DG-Lie algebras $T^\Lscr_{\poly}(\Oscr)$, $D^{\Lscr}_{\poly}(\Oscr)$ are
obtained by sheafifying the presheaves
\begin{align*}
U&\mapsto T^{\Lscr(U)}_{\poly}(\Oscr(U))\\
U&\mapsto D^{\Lscr(U)}_{\poly}(\Oscr(U))
\end{align*}
for $U\in \operatorname{Ob}(\Cscr)$. 
\begin{theorems} \label{ref-6.4.1-89} There is an isomorphism between
  $T^\Lscr_{\poly}(\Oscr)$ and $D^{\Lscr}_{\poly}(\Oscr)$ in
  $\operatorname{HoLieAlg}(\Oscr)$, the homotopy category of sheaves
  of DG-Lie algebras, which induces the HKR-isomorphism on cohomology.
\end{theorems}
\begin{proof}
  We replace $\Cscr$ with the full subcategory consisting of $U\in
  \Cscr$ such that there is an isomorphism $\Lscr\mid U\cong
  (\Oscr\mid U)^d$ (this does not change the category of sheaves).

If $p:U\r V$ is now a map in $\Cscr$ then since
$\Lscr(V)\cong\Oscr(V)^d$, $\Lscr(U)\cong \Oscr(U)^d$ we have that the
restriction morphism
\[
p^\ast:(\Oscr(V),\Lscr(V))\r (\Oscr(U),\Lscr(U))
\]
satisfies the condition \eqref{ref-6.2-67}, i.e.
$\Oscr(U)\otimes_{\Oscr(V)}\Lscr(V)\cong \Lscr(U)$. 

Put ${}^p\mathfrak{l}^{\Lscr}(U)=\mathfrak{l}^{\Lscr(U)} $ where $\mathfrak{l}^{\Lscr(U)}$
is as in Theorem \ref{ref-6.1-65}. Then ${}^p\mathfrak{l}^{\Lscr}$ is a presheaf
of DG-Lie algebras. Let ${}^pT^{\Lscr}_{\poly}(\Oscr)$ and ${}^pD^{\Lscr}_{\poly}(\Oscr)$
be respectively the presheaves of DG-Lie algebras of $\Lscr$-poly-vector
fields and $\Lscr$-poly-differential operators. 

From the commutative diagram \eqref{ref-6.3-68} we now 
deduce the existence of $L_\infty$-quasi-isomorphisms of presheaves
\begin{equation}
\label{ref-6.21-90}
{}^pT^{\Lscr}_{\poly}(\Oscr)\r {}^p\mathfrak{l}^{\Lscr}\l {}^pD^{\Lscr}_{\poly}(\Oscr)
\end{equation}
Let $\mathfrak{l}^{\Lscr}$ be the sheaffification of ${}^p\mathfrak{l}^{\Lscr}$. Sheafifying
\eqref{ref-6.21-90} finishes the proof.
\end{proof}

\section{Atiyah classes and jet bundles}

In this section we relate Atiyah classes to jet bundles. That this is
possible is well known (see e.g.\ \cite[\S 4]{Kapranov2}) although
we could not find the exact result we need (Proposition \ref{ref-7.4.2-107} below)
in the literature. 

\subsection{Reminder}\label{ref-7.1-91}

We define $(\Cscr,\Oscr,\Lscr)$ as in \S\ref{ref-6.4-88}.  Let
$\Escr$ be an arbitrary $\Oscr$-module.  The Atiyah class
$A_{\Lscr}(\Escr)\in \Ext^1_\Oscr(\Escr,\Lscr^\ast\otimes \Escr)$ is
the obstruction against the existence of an $\Lscr$-connection (not
necessarily flat) on $\Escr$.

Let us briefly recall how $A_{\Lscr}(\Escr)$ is constructed. By
\eqref{ref-3.7-19} we have $J^1\Lscr=\Oscr_1\oplus \Lscr^\ast
=\Oscr_2\oplus \Lscr^\ast$ as $\Oscr_1$ and
$\Oscr_2$-algebras.

We consider the short exact sequence of 
$\Oscr_1$-modules
\begin{equation}
\label{ref-7.1-92}
0\r \Lscr^\ast\otimes_{\Oscr} \Escr\r J^1\Lscr\otimes_{\Oscr_2} \Escr\r \Escr\r 0
\end{equation}
The class of this sequence in $\Ext^1_\Oscr(\Escr,\Lscr^\ast\otimes
\Escr)$ is $A_\Lscr(\Escr)$.  To see that this is the obstruction
against the existence of a connection let $\beta_0:\Escr\r
J^1\Lscr \otimes_{\Oscr_2} \Escr$ be the canonical splitting (as sheaves of
abelian groups) of \eqref{ref-7.1-92} obtained from the decomposition
$J^1\Lscr=\Oscr_2\oplus \Lscr$.  Then any splitting $\beta:\Escr\r
J^1\Lscr\otimes_{\Oscr_2} \Escr$ as $\Oscr_1$ modules yields a
connection $\nabla:\Escr\r \Lscr^\ast \otimes_{\Oscr} \Escr$ given by
$\beta-\beta_0$. It is easy to see that this construction is
reversible.

Taking powers and symmetrizing we obtain an element $a(\Escr)^n$ in 
$\Ext^n_{\Oscr}( \Escr,\wedge^n \Lscr^\ast\otimes \Escr)$. The $n$'th
(scalar) Atiyah class $a_n(\Escr)\in H^n(\Cscr,\wedge^n \Lscr^\ast)$
of $\Escr$ is the trace of $a(\Escr)^n$. 

\subsection{Atiyah classes: algebraic background}\label{ref-7.2-93}

We need some functoriality properties of the Atiyah class. To deduce these
cleanly we work in a somewhat more abstract setting and introduce
some adhoc terminology. 

Let $\Sh^{\text{bi}}(\Cscr)$ be the category of sheaves of abelian groups
on $\Cscr$ graded by $\ZZ^2$. If $\Fscr\in \Sh^{\text{bi}}(\Cscr)$
and $f$ is a section of $\Fscr_{i,j}$ then $|f|=i+j$ is the total degree of~$f$. 
As always apply the Koszul sign convention with respect to total degree.

The category
$\Sh^{\text{bi}}(\Cscr)$ is equipped with two obvious shift functors each of total degree one
\begin{align*}
\Fscr[1]_{i,j}&=\Fscr_{i+1,j}\\
\Fscr(1)_{i,j}&=\Fscr_{i,j+1}
\end{align*}
\begin{definitions} \begin{enumerate}
\item A bigraded \emph{DG-algebra} on $\Cscr$ is a bigraded
  sheaf of algebras $\Ascr$ on $\Cscr$ equipped with a derivation $\bar{d}_\Ascr$
  of degree $(1,0)$ such that $\bar{d}_\Ascr^2=0$.
\item A \emph{dDG-algebra $A$} on $\Cscr$ is a bigraded sheaf of
  DG-algebras on $\Cscr$ equipped with an additional derivation
  $d_\Ascr$ of degree $(0,1)$ such that $\bar{d}_\Ascr
  d_\Ascr+d_\Ascr\bar{d}_\Ascr=0$.
\item Assume that $\Ascr$ is a bigraded DG-algebra.  A \emph{DG-$\Ascr$
    module} is a bigraded sheaf of $\Ascr$-modules $\Mscr$ equipped
  with an additive map $\bar{d}_M:\Mscr\r \Mscr$ of degree $(1,0)$
  such that $\bar{d}^2_\Mscr=0$ and such that
\[
\bar{d}_\Mscr(am)=\bar{d}_\Ascr(a)m+(-1)^{|a|}a\bar{d}_\Mscr(m)
\]
for $a,m$ homogeneous sections of $\Ascr$ and $\Mscr$. We denote the category of
DG-modules over $\Ascr$ by $\DGMod(\Ascr)$. 
\item Assume that $\Mscr$ is DG-module over a dDG-algebra $\Ascr$.
  Then a \emph{connection} on $\Mscr$ is an additive map
  $d_\Mscr:\Mscr\r \Mscr$ of degree $(0,1)$ such that
\[
d_\Mscr(am)=d_\Ascr(a)m+(-1)^{|a|}ad_\Mscr(m)
\]
\item The functors $?[1]$ and $?(1)$ change the signs of both $d_\Mscr$ and $\bar{d}_\Mscr$,
when applicable. 
\item Assume that $M$ is a DG-$\Ascr$-module over a dDG-algebra $\Ascr$, equipped with a connection.
  Then the \emph{curvature} of $\Mscr$ is defined as
  $R_\Mscr=-(d_\Mscr\bar{d}_\Mscr+\bar{d}_\Mscr d_\Mscr)$. This is a map $R_\Mscr:\Mscr\r
 \Mscr(1)[1]$ 
in $\DGMod(\Ascr)$. 
\item The derived category of $\DGMod(\Ascr)$, equipped with the shift functor $?[1]$,
is denoted by $D(\Ascr)$. 
\end{enumerate}
\end{definitions}
\begin{examples}
Let $A\to B$ be a morphism of sheaves of commutative DG-algebras. Then $\Omega_{B/A}$ is a 
dDG-algebra. The bigrading comes from the internal (coming from $B$) and external (exterior) degrees. 
The degree $(0,1)$ derivation $d$ is the De Rham differential and the degree $(0,1)$ derivation 
$\bar{d}$ is characterised by the property that it commutes with $d$ and that it coincides with $d_B$ 
on $B=\Omega^0_{B/A}$. 
\end{examples}
Assume that $\Ascr$ is a dDG-algebra. We define a bigraded DG-algebra. 
\[
J^1\Ascr=\Ascr\oplus \Ascr\epsilon
\]
where $\epsilon$ satisfies $\bar{d}_\Ascr(\epsilon)=0=\epsilon^2$, has degree $(0,-1)$ and
\[
a\epsilon=(-1)^{|a|} \epsilon a
\]
We have two algebra morphisms commuting with $\bar{d}_\Ascr$:
\[
i_1:\Ascr\r J^1\Ascr:a\mapsto a
\]
\[
i_2:\Ascr\r J^1\Ascr:a\mapsto a+\epsilon d_\Ascr(a)
\]
We view $J^1\Ascr$ as a DG-$\Ascr$-bimodule via $i_1$, $i_2$. 

We get an associated exact sequence of $\Ascr$-$\Ascr$ bimodules
\begin{equation}
\label{ref-7.2-94}
0\r \Ascr\epsilon\r J^1\Ascr\r \Ascr\r 0
\end{equation}
Let $\Mscr\in \DGMod(\Ascr)$. Tensoring \eqref{ref-7.2-94} on the right by $\Mscr$ we obtain an exact
sequence in $\DGMod(\Ascr)$
\begin{equation}
\label{ref-7.3-95}
0\r \Mscr(1)\r  J^1\Ascr \otimes_{\Ascr} \Mscr\r \Mscr\r 0
\end{equation}
with
\[
\Mscr(1)\r J^1\Ascr\otimes_{\Ascr} \Mscr:n\mapsto \epsilon  \otimes n
\]
\[
J^1\Ascr\otimes_{A} \Mscr\r \Mscr:(a+b\epsilon)\otimes m\mapsto am
\]
\begin{definitions} Let $\Mscr\in \DGMod(\Ascr)$. The \emph{Atiyah class}
  $A(\Mscr)$ of $\Mscr$ is the element of
  $\Hom^1_{D(\Ascr)}(\Mscr,\Mscr(1))$ representing the exact sequence
  \eqref{ref-7.3-95}.
\end{definitions}
\begin{lemmas} \label{ref-7.2.3-96} If $\Mscr$ has a connection then
  $A(\Mscr)=R_\Mscr$. In other words $A(\Mscr)$ is represented by an actual
  map
\[
R_\Mscr:\Mscr\r \Mscr(1)[1]
\]
of bigraded $\Ascr$-modules.
\end{lemmas}
\begin{proof}
If $M$ has
a connection $d_\Mscr$ then the map
\[
\beta:\Mscr\r J^1\Ascr\otimes_{\Ascr} \Mscr:m\mapsto 1\otimes m-\epsilon\otimes d_\Mscr(m)
\]
defines a right splitting of \eqref{ref-7.3-95} as graded $\Ascr$-modules.  The
corresponding left splitting is 
\[
\alpha:J^1\Ascr\otimes_\Ascr \Mscr\r \Mscr(1):1\otimes m+\epsilon
\otimes n\mapsto d_\Mscr(m)+n\qed
\]
Since \eqref{ref-7.3-95} is split its corresponding class in
$\Hom^1_{D(\Ascr)}(\Mscr,\Mscr(1))$ is given by\footnote{To see this
  one should think of a degreewise  split exact sequence as a shift to
  the left of a standard triangle constructed from a mapping cone.
See \cite[I\S2]{RD}.
}
$-\alpha \bar{d}_\Mscr\beta$. One computes that this is equal to
$R_\Mscr$.
\end{proof}
\begin{lemmas} \label{ref-7.2.4-97} The Atiyah class is functorial in the
  following sense.  Assume that we have a morphism of dDG-algebras
  $\theta:\Ascr\r \Bscr$ and DG-modules $\Mscr$, $\Nscr$ over $\Ascr$ and $\Bscr$ as well as an
  additive map $\psi:\Mscr\r \Nscr$ of degree zero which is compatible with
  the differentials and $\theta$ in the sense that 
$\psi\circ \bar{d}_\Mscr=\bar{d}_\Nscr\circ \phi$ and
$\psi(am)=\theta(a)\psi(m)$.
Then the following diagram is commutative in $D(\Ascr)$
\[
\begin{CD} 
\Mscr @>A(\Mscr)>> \Mscr(1)[1]\\
@V\psi VV @VV\psi V\\
{}_\Ascr{\Nscr} @>>{}_\Ascr(A(N))> {}_\Ascr\Nscr(1)[1]
\end{CD}
\]
\end{lemmas}
\begin{proof} This follows from the functoriality of the exact sequence
\eqref{ref-7.3-95}.
\end{proof}

\subsection{Scalar Atiyah classes}

Let $\Ascr$ be a dDG-algebra on $\Cscr$ and let $\Mscr\in
\DGMod(\Ascr)$. We assume in addition that $\Mscr$ is locally free of
constant rank $e$ over $\Cscr$. I.e. the topology on $\Cscr$ has a
basis $\Bscr$ such that for $U\in \Bscr$ we have that $\Mscr_U\cong
\Ascr_U^{\oplus e}$ as bigraded $\Ascr$-modules. We may now view
$A(\Mscr)^n$ as an element of $\Hom^n_{D(\Ascr)}(\Mscr,\Mscr(n))$, or
since $\Mscr$ is locally free, as an element of
\[
\mathbb{H}^n(\Cscr,\HEnd_{\Ascr}(\Mscr)_{\ast,n})
\]
where $\mathbb{H}$ denotes hypercohomology and $\HEnd_{\Ascr}(\Mscr)_{\ast,n}$
is equipped with the differential $[d_{\Mscr},-]$.

It is easy to check
locally that the trace map
\[
\Tr :\HEnd_\Ascr(\Mscr)\r \Ascr
\]
is in $\DGMod(\Ascr)$. Thus we obtain a map on hypercohomology
\[
\Tr:\mathbb{H}^n(\Cscr,\HEnd_{\Ascr}(\Mscr)_{\ast,n})\r
\mathbb{H}^n(\Cscr,\Ascr_{\ast,n})
\]
We call 
\[
a_n(\Mscr)=\Tr(a(\Mscr)^n)\in \mathbb{H}^n(\Cscr,\Ascr_{\ast,n})
\]
the $n$'th (scalar) Atiyah class of $M$.
\begin{lemmas} 
\label{ref-7.3.1-98}
Assume that we have a morphism $\theta:\Ascr\r \Bscr$ of dDG-algebras
and assume that $\Mscr\in \DGMod(\Ascr)$ is locally free of rank $e$.
Put $\Nscr=\Bscr\otimes_\Ascr \Mscr$. Then $\Nscr\in \DGMod(\Bscr)$ is locally free of rank
$e$. We have
\[
a_n(\Nscr)=\mathbb{H}^n(\theta)(a_n(\Mscr))
\]
where $\mathbb{H}^n(\theta)$ is the natural map
\[
\mathbb{H}^n(\theta):\mathbb{H}^n(\Cscr,\Ascr_{\ast,n})\r 
\mathbb{H}^n(\Cscr,\Bscr_{\ast,n})
\]
\end{lemmas}
\begin{proof}
The commutative diagram
\[
\begin{CD}
\Mscr @>A(\Mscr)^n >> \Mscr[n](n)\\
@VVV @VVV\\
\Nscr @>A(\Nscr)^n >> \Nscr[n](n)
\end{CD}
\]
obtained from Lemma \ref{ref-7.3.1-98} 
may be translated into saying that $A(\Nscr)^n$ is the image of $A(\Mscr)^n$
under the induced map
\[
\HH^n(\Cscr,\HEnd_{\Ascr}(\Mscr)_{\ast,n})
\xrightarrow{\Bscr\otimes_\Ascr-}
\HH^n(\Cscr,\HEnd_{\Bscr}(\Nscr)_{\ast,n})
\]
One verifies locally that there is a commutative diagram of DG-modules
\[
\begin{CD}
\HEnd_{\Ascr}(\Mscr)@>\Bscr\otimes_\Ascr->> \HEnd_{\Bscr}(\Nscr)\\
@V \Tr VV @VV \Tr V\\
\Ascr@>>> \Bscr
\end{CD}
\]
This finishes the proof. 
\end{proof}
\begin{examples}
\label{ref-7.3.2-99} We explain how the Atiyah class  constructed in
\S\ref{ref-7.1-91} fits into this framework. 
 We define $\Ascr$ as the De Rham complex $(\wedge
  \Lscr^\ast_1,d)$  and put it in degrees $(0,\ast)$. We define $\Mscr=
\Ascr\otimes_{\Oscr} \Escr$. 
 
The Atiyah class $A_\Lscr(\Escr)\overset{\text{def}}{=}A(\Mscr)$
now becomes an element of
\begin{align*}
\Ext^1_{\gr(\Ascr)}(\Mscr,\Mscr(1))&=
\Ext^1_{\gr(\Ascr)}(\Ascr\otimes_{\Oscr} \Escr,(\Ascr\otimes_{\Oscr} \Escr)(1))\\
&=\Ext^1_{\gr(\Oscr)}(\Escr,(\Ascr\otimes_{\Oscr} \Escr)(1))\\
&=\Ext^1_{\Oscr}(\Escr,\Lscr^\ast\otimes_{\Oscr} \Escr)
\end{align*}
It is easy to see that
$A_\Lscr(\Escr)\in\Ext^1_{\Oscr}(\Escr,\Lscr^\ast\otimes_{\Oscr} \Escr)$
represents the part of degree zero of \eqref{ref-7.3-95}. This is
\[
0\r \Lscr^\ast\otimes_{\Oscr} \Escr\r J^1\Lscr\otimes_{\Oscr} \Escr\r \Escr\r 0
\]
Hence $A_\Lscr(\Escr)$ coincides with our previous definition. It is easy to
deduce from this that we also get the same $a_{n,\Lscr}(\Escr)$.
\end{examples}

\subsection{Atiyah classes from jet bundles}

We assume we are in the setting from \S\ref{ref-7.1-91}. 
As outlined in the previous section we will work with bigraded sheaves. 

We first consider the $\Lscr_2$-De Rham complex $\wedge \Lscr_2^\ast$
as a dDG-algebra concentrated in degrees $(0,*)$ with $\bar{d}=0$.  We
then let $C$ be a commutative $DG$-algebra such that $JL$ is equipped
with a flat $C$-connection $\nabla$. Thus
$(C\ctimes_{\Oscr_1}J\Lscr,\nabla)$ becomes a DG-algebra (actually a
DG-$C$-algebra).  From Lemma \ref{lemma3.3.6} we obtain a 
morphism of dDG-algebras 
$$
\theta:(\wedge \Lscr_2^\ast,\bar{d}=0)\r
\Omega_{C\ctimes_{\Oscr_1}J\Lscr/C} $$
If we set
$\Mscr=((\wedge\Lscr_2^\ast)\otimes_{\Oscr_2}\Lscr,\bar{d}=0)\in \DGMod(\wedge\Lscr_2^\ast,\bar{d}=0)$ then we obtain 
from Lemma \ref{ref-7.3.1-98} and Example \ref{ref-7.3.2-99} 
\begin{equation}\label{ref-7.8-104}
\mathbb{H}(\theta)(a_{n,\Lscr}(\Lscr))=\mathbb{H}(\theta)(a_n(\Mscr))
=a_n(\Omega_{C\ctimes_{\Oscr_1}J\Lscr/C}\ctimes_{\wedge\Lscr_2^\ast}\Mscr)
=a_n(\underbrace{\Omega_{C\ctimes_{\Oscr_1}J\Lscr/C}\ctimes_{\Oscr_2}\Lscr}_{\overset{\text{def}}{=}\Nscr})\,.
\end{equation}
Finally recall that $(C\ctimes_{\Oscr_1} J\Lscr)\ctimes_{\Oscr_2}\Lscr_2\cong\HDer_C(C\ctimes_{\Oscr_1}J\Lscr)$, 
therefore 
$$
\Nscr\cong(\Omega_{C\ctimes_{\Oscr_1}J\Lscr/C})\ctimes_{C\ctimes_{\Oscr_1}J\Lscr}
\underbrace{\HDer_C(C\ctimes_{\Oscr_1}J\Lscr)}_{\overset{\text{def}}{=}\Nscr_0})\,.
$$
The fact that $\HH(\theta)(a_n(\Lscr))=a_n(\Nscr)$ provides a mean to compute $a_n(\Lscr)$ if we 
can compute $a_n(\Nscr)$. The latter can be accomplished if we can put a connection on $\Nscr$ 
(see Lemma \ref{ref-7.2.3-96})
\begin{lemmas}\label{ref-7.4.1-105}
Assume that there is an isomorphism of graded $C$-algebras
\begin{equation}\label{ref-7.9-106}
\pi:C\ctimes_{\Oscr_1} J\Lscr\r C\ctimes_{\Oscr} \widehat {S\Lscr^\ast}
\end{equation}
which induces the identity map $C\otimes_{\Oscr}\gr J\Lscr\cong C\otimes_{\Oscr} S\Lscr^\ast$. 
Then $\Nscr$, as introduced above, has a connection.
\end{lemmas}
\begin{proof}
We use the isomorphism \eqref{ref-7.9-106} to transport the differential
$\nabla$ (defined on $B\overset{\text{def}}{=}C\ctimes_{\Oscr}J\Lscr$) to a 
differential $\tilde\nabla$ on 
$\tilde B\overset{\text{def}}{=}C\ctimes_{\Oscr}\widehat{S\Lscr^\ast}$. 
Note that this differential does not have a simple expression.
As for  $\nabla$, we extend $\tilde\nabla$ to a unique differential on 
$\Omega_{\tilde B/C}\cong C\ctimes\Omega_{\widehat{S\Lscr^\ast}/\Oscr}$ in such a way 
that it commutes with the De Rham differential $\tilde d_{{\rm DR}}$. 

We now put 
\[
\tilde{\Nscr}_0=\HDer_C(\tilde{B})\cong \tilde B\ctimes_{\Oscr}\Lscr
\quad\textrm{and}\quad
\tilde{\Nscr}=\Omega_{\tilde{B}/C}\ctimes_{\tilde{B}}\tilde{\Nscr}_0=\Omega_{\tilde{B}/C}\ctimes_{\Oscr}\Lscr\,.
\]
The isomorphism \eqref{ref-7.9-106} between $B$ and $\tilde{B}$ yields isomorphims
between $\Omega_{B/C}$ and $\Omega_{\tilde{B}/C}$, between $\Nscr_0$ and $\tilde{\Nscr}_0$ 
and between $\Nscr$ and $\tilde{\Nscr}$. 
We can now define a connection on $\tilde{\Nscr}=\Omega_{\tilde{B}/C}\otimes_{\Oscr}\Lscr$ by 
putting
\[
d_{\tilde{\Nscr}}(b\otimes l)=\tilde d_{{\rm DR}}(b)\otimes l
\]
It is easy to see that this is well-defined. Transporting accross the isomorphism 
$\Nscr\cong \tilde{\Nscr}$ yields a connection $d_{\Nscr}$ on $\Nscr$.
\end{proof}

The DG-algebras $C^{\coord,\Lscr}$ and $C^{\aff,\Lscr}$ are equipped with a canonical
map $\wedge \Lscr^\ast\r C^{\coord,\Lscr}$ and $\wedge \Lscr^\ast\r C^{\aff,\Lscr}$  
as follows from the definitions in \S\ref{ref-4.2.1-36} and \S\ref{ref-4.3-43}.

In addition condition \eqref{ref-7.9-106} applies with $C=C^{\aff,\Lscr}$ and 
$C=C^{\coord,\Lscr}$, see \eqref{ref-4.11-48}. We have natural morphisms
\[
\wedge\Lscr_2^\ast\xrightarrow{\theta}C^{\aff,\Lscr}\ctimes_{\Oscr_1} \Omega_{J\Lscr/\Oscr_1}
\xrightarrow{\psi}C^{\coord,\Lscr}\ctimes_{\Oscr_1} \Omega_{J\Lscr/\Oscr_1}
\xrightarrow[\cong]{\mu} C^{\coord,\Lscr}\ctimes \Omega_{F/k}
\]
Below we decorate notations referring to $C^{\aff,\Lscr}$ and $C^{\coord,\Lscr}$
by  superscripts ``$\aff$'' and ``$\coord$'' respectively.
For example we define $\Nscr^{\aff}$ and $\Nscr^{\coord}$ like $\Nscr$ in the above 
discussion where we replace $C$ by $C^{\aff,\Lscr}$ and $C^{\coord,\Lscr}$.
\begin{propositions}\label{ref-7.4.2-107}
Write the Maurer-Cartan form (see \eqref{ref-4.5-38}) as
\[
\omega=\sum_{i,\alpha} \eta_\alpha\omega^i_\alpha\partial_i
\]
with $\eta_\alpha\in (C^{\coord,L})_1$, $\omega^i_\alpha\in F$.
Then we have as elements of complexes
\begin{align*}
(\mu\psi)(a_n(\Nscr^{\aff}))=\Tr(\Xi^n)
\end{align*}
where $\Xi$ is the matrix with entries 
\[
\sum_\alpha\eta_{\alpha} d_F(\partial_j\omega^i_\alpha)
\]
Furthermore as cohomology classes we have
\begin{equation}\label{ref-7.10-108}
\HH(\theta)(a_{n,\Lscr}(\Lscr))=a_n(\Nscr^{\aff})
\end{equation}
\end{propositions}
\begin{proof}
The identity \eqref{ref-7.10-108} is \eqref{ref-7.8-104}.

We can use the canonical connection on $\Nscr^{\aff}$ and $\Nscr^{\coord}$
exhibited in the proof of Lemma \ref{ref-7.4.1-105} to compute
$a_n(\Nscr^{\aff})$ and $a_n(\Nscr^{\coord})$ (using Lemma \ref{ref-7.2.3-96}).
Since these connections are compatible we get
\[
a_n(\Nscr^{\coord})=\psi(a_n(\Nscr^{\aff}))
\]
as elements of complexes.

We now compute $a_n(\Nscr^{\coord})$ explicitly. We
have identifications (e.g.\ \eqref{ref-4.12-49} and \eqref{ref-4.13-50})
\[
C^{\coord,\Lscr}\ctimes_{\Oscr_1} J\Lscr
\cong C^{\coord,\Lscr}\ctimes_{\Oscr} 
\widehat {S\Lscr}_1\cong C^{\coord,L}\ctimes F
\]
Using these identifications we have
\begin{align*}
\Nscr^{\coord}&=C^{\coord,\Lscr}\ctimes_{\Oscr} 
\Omega_{\widehat {S\Lscr^\ast}/\Oscr}\ctimes_{\Oscr} \Lscr
=C^{\coord,\Lscr}\ctimes \Omega_F\ctimes \sum_i k\partial_i
\end{align*}
where $\partial_i=\partial/\partial t_i$. 
We have
\[
\partial_1,\ldots,\partial_d\in
\Oscr^{\coord,L}\ctimes_{\Oscr} \Lscr\subset
C^{\coord,\Lscr}\ctimes_{\Oscr} \Omega_{\widehat
  {S\Lscr^\ast}/\Oscr}\ctimes_{\Oscr} \Lscr
\]
and since $d_{\Nscr^{\coord}}$ is zero on $\Oscr^{\coord,\Lscr}\ctimes_{\Oscr} \Lscr$
we deduce
\[
d_{\Nscr^{\coord}}(\partial_i)=0
\]
For further computation we use the identification
\[
\Nscr^{\coord}=C^{\coord,\Lscr}\ctimes \Omega_F\ctimes \sum_i k\partial_i
\]
where $d_{\Nscr^{\coord}}$ acts as
\[
d_{\Nscr^{\coord}}(c\ctimes \omega\ctimes \partial_i)=
(-1)^{|c|}c\ctimes d_F\omega
\ctimes \partial_i
\]
Remember from \S \ref{ref-4.2.1-36} that the differential $\bar{d}_{B^{\coord}}={}^1\nabla^{\coord}$ 
on $C^{\coord,\Lscr}\ctimes \Omega_F\cong\Omega_{B^{\coord}/C^{\coord,\Lscr}}$ is given by 
\[
d_{C^{\coord,\Lscr}}\otimes 1+\sum_{i,\alpha}\eta_\alpha\omega^i_\alpha \, \partial_i
\]
where we think of $\partial_i$ as a Lie derivative. We compute
\begin{align*}
\bar{d}_{\Nscr^{\coord}}(\partial_j)&=[d_{C^{\coord,L}}\otimes 1+\sum_{i,\alpha}
\eta_\alpha\omega^i_\alpha
\, {\partial_i},\partial_j]\\
&=\sum_{i,\alpha}(\eta_\alpha\partial_j\omega^\alpha_i)\partial_i
\end{align*}
and hence
\begin{align*}
R_{\Nscr^{\coord}}(\partial_j)&=
-(d_{\Nscr^{\coord}}\bar{d}_{\Nscr^{\coord}}+d_{\Nscr^{\coord}}\bar{d}_{\Nscr^{\coord}})(\partial_j)\\
&=-d_{\Nscr^{\coord}}(\sum_{i,\alpha}\eta_\alpha(\partial_j\omega^i_\alpha)\partial_i)\\
&=\sum_{i,\alpha} \eta_{\alpha} d_F(\partial_j\omega^i_\alpha)\partial_i
\end{align*}
Thus
$
\mu(a_n(\Nscr^{\coord}))=\Tr(\Xi^n)
$
where
$\Xi$ is as in the statement of the proposition. This finishes the proof. 
\end{proof}
\section{The Kontsevich local formality quasi-isomorphism}
\label{ref-8-109}
\subsection{The $L_\infty$-morphism}
In this section we assume that $k$ contains the reals and we describe
the exact form of the Kontsevich local formality morphism.

As above let $F=k[[t_1,\ldots,t_d]]$ and $T_{\poly}(F)$, $D_{\poly}(F)$
are respectively the Lie algebras of poly-vector fields and poly-differential
operators over $F$. We equip $T_{\poly}(F)$ and $D_{\poly}(F)$ with the
shifted Gerstenhaber structures introduced in \S\ref{ref-3.2.2-11}. 
For $\gamma\in T_{\poly}^n(F)$ we put
\[
\gamma^{i_1,\ldots, i_{n+1}}=\langle dt_{i_1}\wedge\cdots
\wedge dt_{i_{n+1}},\gamma\rangle
\]
where 
$
\langle -,-\rangle 
$
is the pairing introduced in \eqref{ref-3.2-13}.

The Kontsevich local formality isomorphism $\Uscr:T_{\poly}(F)\r
D_{\poly}(F)$ is defined as follows. We put
\[
\Uscr_n=\sum_{m\ge 0} \sum_{\Gamma\in G_{n,m}} W_\Gamma \Uscr_\Gamma
\]
where the $W_\Gamma$ are some coefficients to be defined below and
where $G_{n,m}$ is a set of directed graphs $\Gamma$ described as follows
\begin{enumerate}
\item There are $n$ vertices of the ``first type'' labeled by $1,\ldots,n$. 
\item There are $m$ vertices of the ``second type''
  labeled by $1,\ldots,m$.
\item The vertices of the second type have no outgoing arrow.
\item There are no loops and double arrows.
\item There are $2n+m-2$ edges. 
\item All edges carry a distinct label..
\end{enumerate}
For use below we also introduce $G_{n,m,\epsilon}$ which is defined in the
same way except that the number of edges of the graphs should be equal to
$2n+m-2-\epsilon$. The number of edges in a graph is denoted by $|\Gamma|$. 

For a vertex $v$ of $\Gamma$ we denote the incoming and outgoing edges of $v$
by $\text{In}(v)$ and $\text{Out}(v)$ respectively. 
Let $\Gamma_i$ be the vertices of the $i$'th kind for $i=1,2$.  Let $\gamma_i\in
T_{\poly}(F)$ and 
put $k_i=|\gamma_i|$.
By definition $\Uscr_\Gamma(\gamma_1\cdots \gamma_n)$ is zero unless
$|\text{Out}(i)|=k_i+1$. In that case
\[
\Uscr_\Gamma(\gamma_1\cdots \gamma_n)(f_1\cdots f_m)=
\prod_{%
\begin{smallmatrix}
v\in \Gamma_1\\
\text{In}(v)={r_1,\ldots,r_d}\\
\text{Out}(v)={s_1,\ldots,s_{k_v+1}}
\end{smallmatrix}
}
\partial_{r_1}\cdots \partial_{r_d}\gamma_v^{s_1\cdots \gamma_{k_v+1}}
\prod_{%
\begin{smallmatrix}
v\in \Gamma_2\\
\text{In}(v)={t_1,\ldots,t_e}
\end{smallmatrix}
}
\partial_{t_1}\cdots \partial_{t_e} f_v
\]
where we assume that the ordering on the labels $s_1\cdots s_d$ is such
that $s_1<\cdots< s_{k_v+1}$. 

The coefficients $W_\Gamma$ are defined as integrals over configuration
spaces. 
 Let $\Hscr$ be the upper half plane and let $\RR$ be its horizontal
boundary.  The group 
\[
G^{(1)}=\{z\mapsto az+b\mid a,b\in \RR, a>0\}
\]
acts on $\Hscr\cup \RR$. $C^+_{n,m}$ is the quotient
$\text{Conf}^+_{n,m}/G^{(1)}$ where $\text{Conf}^+_{n,m}$ is the space
of configurations of $n$ distinct points $p_1,\ldots,p_n$ in $\Hscr$
and $m$ distinct points $q_1,\ldots,q_m$ in $\RR$ such that
$q_1<\cdots < q_m$. The manifold $C^+_{n,m}$ will be oriented as follows (see \cite{AMM1}). One puts $p_1$ in a fixed position and uses
the coordinates of the other points to identify $C^+_{n,m}$ with an open
subset of the affine space $\AA=\CC^{n-1}\times \RR^m$. One then transfers
the standard orientation on $\AA$ to $C^+_{n,m}$. 

For $\Gamma\in G_{n,m,\epsilon}$ put 
\[
\kappa_\Gamma=\bigwedge_{e\in \{\text{edges of $\Gamma$}\}} d\phi_e
\]
 where the (multi-valued) function $\phi_e$ on
$C^+_{n,m}$ is defined as $\tilde{\phi}_e/2\pi$ where $\tilde{\phi}_e$
is computed as in the following image
\begin{center}
\psfrag{h}[l][l]{$\tilde{\phi}_e$}
\psfrag{p}[r][r]{$p$}
\psfrag{q}[t][t]{$q$}
\psfrag{e}[t][t]{$e$}
\includegraphics[width=7cm]{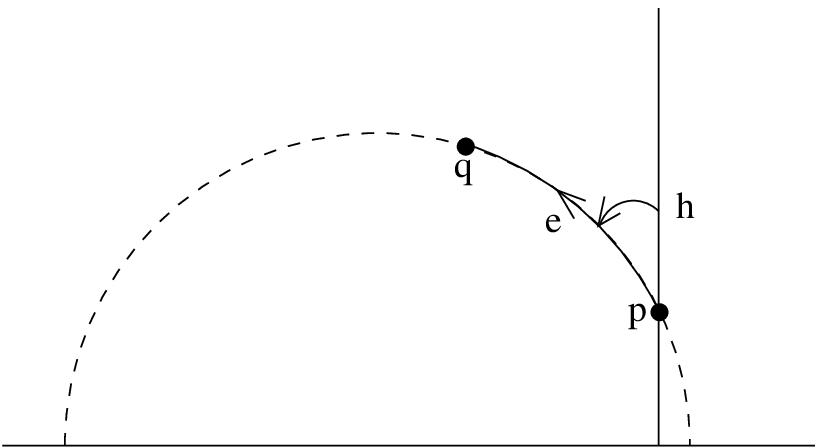}
\end{center}
Thus if $e$ is an edge in $\Gamma$ from $p$ to $q$ then we embed $e$ as a
line in the hyperbolic plane $\Hscr$ and we measure the angle $\phi_e$ in the counter clockwise direction as
indicated in the drawing.

The ordering of the edges in the product $\bigwedge_{e} d\phi_e$ is
first according to the ordering of the starting vertices in the set
$\Gamma_1$ (the ordering is by label) and then according to an
arbitrarily chosen but fixed ordering on outgoing edges.

Now we put\footnote{This definition
  differs by a sign from Kontsevich's definition. Kontsevich's
  definition necessitates an unpleasant sign change in the definition of
the Lie bracket
  on $T_{\poly}(F)$ (see \cite{AMM1}). Moreorer this sign change
  destroys the Gerstenhaber property of $T_{\poly}(F)$. A tedious
  computation shows that with our definition no sign changes for the Lie 
  bracket are necessary. }
\begin{equation}
\label{ref-8.1-110}
W_\Gamma=(-1)^{|\Gamma|(|\Gamma|-1)/2}
\int_{C^+_{n,m}}\kappa_\Gamma
\end{equation}
One easily sees  that the product $W_\Gamma \Uscr_\Gamma$ is
independent of the chosen ordering on outgoing edges. 

Assume $n=1$. In that case $G_{1,m}$ contains only one graph $\Gamma_0$ and
$\kappa_{\Gamma_0}=(-1)^{m(m-1)/2}1/m!$.  Furthermore
\begin{align*}
\Uscr_{\Gamma_0}(\gamma)(f_1,\ldots,f_m)&=\gamma^{i_1\cdots i_m} \partial_{i_1}f_1
\cdots \partial_{i_m}f_m\\
&=\langle df_1\cdots df_m,\gamma\rangle
\end{align*}
where $\langle-,-\rangle$ is as in \eqref{ref-3.2-13}.  Hence
\begin{equation}
\label{ref-8.2-111}
\Uscr_1(\gamma)(f_1,\ldots,f_m)=(-1)^{m(m-1)/2}\frac{1}{m!}\langle df_1\cdots df_m,\gamma\rangle
\end{equation}
It is easy to see that $\Uscr_1$ coincides with the HKR map 
$
T_{\poly}(F)\r D_{\poly}(F)
$
as defined by \eqref{ref-6.5-71}. 

\subsection{Compatibility with cupproduct}

Let $C$ be a commutative DG-$k$-algebra (equipped with some topology). Then we may extend $\Uscr$ to an
$L_\infty$-morphism 
\[
\tilde{\Uscr}:C\ctimes T_{\poly}(F)\r C\ctimes D_{\poly}(F)
\]
Assume now that we have a solution
$\omega=\sum_{\alpha}\eta_\alpha\omega_\alpha$ to the Maurer-Cartan
equation in $(C^1\ctimes T^0_{\poly}(F))$. By property (P4) we know
that $\Uscr_1(\omega)$ is a solution to the Maurer Cartan equation in
$C^1\ctimes D^0_{\poly}(F))$. It will be convenient to denote
$\Uscr_1(\omega)$ also by~$\omega$.

Under suitable convergence hypotheses we know that there exists a
twisted $L_\infty$-morphism (see \S\ref{ref-6.2-69}) 
\[
\tilde{\Uscr}_\omega:(C\ctimes T_{\poly}(F))_\omega\r (C\ctimes D_{\poly}(F))_\omega
\]
Kontsevich sketches a proof that $\Uscr_\omega$ commutes with cup
product up to homotopy.\footnote{Kontsevich's proves this in fact for
  general solutions of the Maurer-Cartan equation in $C\ctimes
  T_{\poly}(F)$.}  A more detailed proof in the slightly restricted 
case that $\omega\in T^1_{\poly}(F)$ (i.e.\ a Poisson bracket) was given in \cite{MT}.  In
\cite{Mochizuki} it is even shown that $\tilde{\Uscr}_\omega$ can be
extended to an $A_\infty$-morphism (this is again in the case
$\omega\in T^1_{\poly}(F))$.

For the benefit of the reader we will state a result below which will
be sufficient for the sequel. It can be obtained by copying the proof of 
\cite{MT}, taking into account our modified sign conventions.  

It is well known that $C^+_{n,m}$ can be canonically compactified as a
manifold with corners $\bar{C}^+_{n,m}$ \cite{Ko3}.  Let
 $\bar{C}_{2,0}=\bar{C}^+_{2,0}$ (the ``$+$'' is superfluous) be the ``Eye'' as in the following figure
\[
\includegraphics[width=7cm]{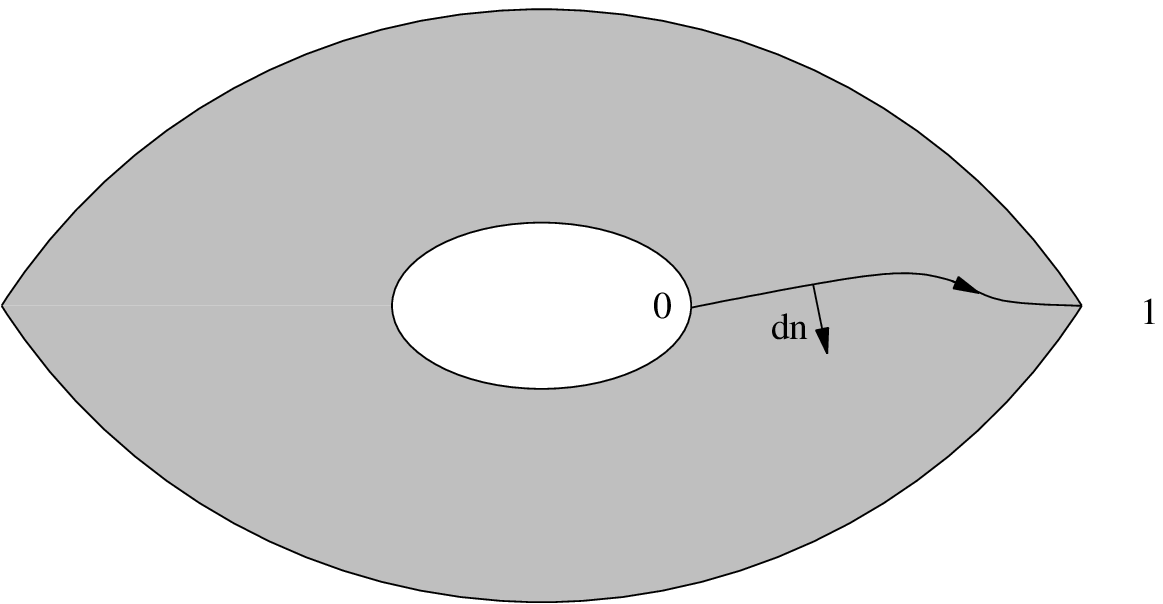}
\]
The upper outer boundary is where $p_1$ approaches the real line, the lower outer boundary
is where $p_2$ approaches the real line. The inner boundary is where
$p_1$ and $p_2$ approach each other, away from the real line. 

The right corner is the locus where $p_1$, $p_2$ both approach the
real line with $p_1$ to the left of $p_2$.  Following Kontsevich
\cite{Ko3} we have indicated a path $\xi:[0,1]\r \bar{C}^+_{2,0}$ from
a point on the inner boundary (labeled ``$0$'') to the right corner
(labeled ``$1$'').

Next we consider the map
\[
F:C^+_{n,m}\r C_{2,0}
\]
which is given by projection onto the first two points. One may show
that this map can be extended to a map 
\[
\bar{F}:\bar{C}^+_{n,m}\r \bar{C}_{2,0}
\]
and we put $Z_{n,m}=\bar{F}^{-1}\xi([0,1])$. We orient $Z$ by the normal $dn$
(as indicated in the above figure). 

For $\Gamma\in G_{n,m,1}$ we put
\[
\tilde{W}_\Gamma=(-1)^{|\Gamma|(|\Gamma|-1)/2}\int_{Z_{n,m}} \kappa_\Gamma
\]
After a tedious computation, mimicking \cite{MT}, we obtain the following.
\begin{proposition} 
\label{ref-8.1-112}
For $\alpha,\beta\in T_{\poly}(F)$ we have as maps from $(C\ctimes T_{\poly}(F))^2_\omega$
to $(C\ctimes D_{\poly}(F))_\omega$
\[
\tilde{\Uscr}_{\omega,1}(\alpha)\cup \tilde{\Uscr}_{\omega,1}(\beta)
-\tilde{\Uscr}_{\omega,1}(\alpha\cup\beta)
+d(H(\alpha,\beta))-H(d\alpha,\beta)-(-1)^{|\alpha|+1}H(\alpha,d\beta)=0
\]
where
\begin{equation}
\label{ref-8.3-113}
H(\alpha,\beta)=\sum_{n,m\ge 0,\Gamma\in G_{n,m,1}}
(-1)^{m-1} \frac{1}{(n-2)!}\tilde{W}_\Gamma
\Uscr_\Gamma(\alpha\beta\omega^{n-2})
\end{equation}
(the operators $\Uscr_\Gamma$ are extended multilinearly to $C\ctimes
T_{\poly}(F)$).
In particular $\tilde{\Uscr}_{\omega,1}$ commutes with cupproduct, up to a natural
homotopy.
\end{proposition}
\section{Proof of Theorem \ref{ref-1.3-0}}
\label{ref-9-114}
In this section we will initially assume that $k$ contains the reals
and we let the local formality morphism $\Uscr$ in \eqref{ref-6.4-70}
be the one defined by Kontsevich (as in \S\ref{ref-8-109}).
\subsection{The local case}
\label{ref-9.1-115}
Combining the $L_\infty$-morphisms 
\eqref{ref-6.11-78}\eqref{ref-6.9-76}\eqref{ref-6.10-77}%
\eqref{ref-6.12-79}\eqref{ref-6.15-82}\eqref{ref-6.13-80} we obtain a
commutative diagram
\begin{equation}
\label{ref-9.1-116}
\xymatrix{
\scriptstyle T_{\poly}^{L_2}(R_2)\ar[r] &
\scriptstyle T_{\poly,C^{\aff,L}}(C^{\aff,L}\ctimes_{R_1} JL)\ar[r]
\ar[d]^{\Vscr^{\aff}} &
\scriptstyle T_{\poly,C^{\coord,L}}(C^{\coord,L}\ctimes_{R_1} JL) \ar[r]^-\cong\ar[d]^{\Vscr^{\coord}}&
\scriptstyle (C^{\coord,L}\ctimes T_{\poly}(F))_{\omega}
\ar[d]^{\tilde{\Uscr}_\omega}\\
\scriptstyle D_{\poly}^{L_2}(R_2)\ar[r] &
\scriptstyle D_{\poly,C^{\aff,L}}(C^{\aff,L}\ctimes_{R_1} JL)\ar[r] &
\scriptstyle D_{\poly,C^{\coord,L}}(C^{\coord,L}\ctimes_{R_1} JL) \ar[r]_-\cong&
\scriptstyle (C^{\coord,L} \ctimes D_{\poly}(F))_{\omega}
}
\end{equation}
where the horizontal maps are strict morphisms (i.e.\ the only the
first Taylor coefficient is non-zero).

Our aim is to sheafify diagram \eqref{ref-9.1-116} and to look at the
result in the derived category of $\Oscr$-modules. To determine the
result it is sufficient to understand the $(-)_1$ part of
\eqref{ref-9.1-116}. I.e.
\[
\xymatrix{
\scriptstyle T_{\poly}^{L_2}(R_2)\ar[r] &
\scriptstyle T_{\poly,C^{\aff,L}}(C^{\aff,L}\ctimes_{R_1} JL)\ar[r]
\ar[d]^{\Vscr^{\aff}_1} &
\scriptstyle T_{\poly,C^{\coord,L}}(C^{\coord,L}\ctimes_{R_1} JL) 
\ar[r]^-\cong\ar[d]^{\Vscr^{\coord}_1}&
\scriptstyle (C^{\coord,L}\ctimes T_{\poly}(F))_{\omega}
\ar[d]^{\tilde{\Uscr}_{\omega,1}}\\
\scriptstyle D_{\poly}^{L_2}(R_2)\ar[r] &
\scriptstyle D_{\poly,C^{\aff,L}}(C^{\aff,L}\ctimes_{R_1} JL)\ar[r] &
\scriptstyle D_{\poly,C^{\coord,L}}(C^{\coord,L}\ctimes_{R_1} JL) \ar[r]_-\cong&
\scriptstyle (C^{\coord,L} \ctimes D_{\poly}(F))_{\omega}
}
\]
\begin{lemmas}
\label{ref-9.1.1-117}
The map $\Vscr^{\aff}_1:  T_{\poly,C^{\aff,L}}(C^{\aff,L}\ctimes_{R_1} JL)
\r D_{\poly,C^{\aff,L}}(C^{\aff,L}\ctimes_{R_1} JL)$ commutes with the Lie bracket
and the cupproduct up to homotopies which are functorial for algebraic
Lie algebroid morphisms which satisfy \eqref{ref-6.2-67}. 
\end{lemmas} 
\begin{proof} For the Lie bracket this is clear since $\Vscr^{\aff}_1$ is obtained
from a $L_\infty$-morphism.

For the cupproduct we need to show that the homotopy $H$ defined by
\eqref{ref-8.3-113} descends to a map
$T_{\poly,C^{\aff,L}}(C^{\aff,L}\ctimes_{R_1} JL)^2\r
D_{\poly,C^{\aff,L}}(C^{\aff,L}\ctimes_{R_1} JL)$. This is a
computation similar to the proof of Proposition \ref{ref-5.4.1-63}
combined with \eqref{ref-4.6-40}.  We need the the following version
of (P5).
\begin{itemize}
\item
$\tilde{W}_\Gamma \Uscr_\Gamma(\gamma\alpha)=0$ for $q\ge 3$ ($q$ being the number edges 
of the ``first type'' in $\Gamma$) and $\gamma\in\frak{gl}_d(k)\subset T^{\text{poly},1}(F)$.
\end{itemize}
This is proved in exactly the same way as (P5). See
\cite[\S7.3.3.1]{Ko3}.
\end{proof}
We will now evaluate the formula  for $\tilde{\Uscr}_{\omega,1}(\gamma)$
where we assume $\gamma\in T_{\poly}(F)$. 
\[
\tilde{\Uscr}_{\omega,1}(\gamma)=\sum_{j\ge 0} \frac{1}{j!} \tilde{\Uscr}_{j+1}
(\omega^j\gamma)
\]
We may write
\[
\omega=\sum_\alpha\eta_\alpha\omega_\alpha
\]
with $\eta_\alpha\in  C^{\coord}_1$ and $\omega_\alpha\in T^0_{\poly}(F)$. Below
we suppress the summation sign over $\alpha$.  Thus
\[
\tilde{\Uscr}_{j+1}(\omega^j \gamma)= \eta_{\alpha_j}\cdots \eta_{\alpha_1}
\Uscr_{j+1}(\omega_{\alpha_1}\cdots \omega_{\alpha_j} \gamma)
\]
To understand the (absence of) signs in this formula we note that we consider
$\tilde{\Uscr}$ as a degree zero map $S^{j+1}((C^{\coord,L}\ctimes T_{\poly}(F))[1])
\r (C^{\coord,L}\ctimes D_{\poly}(F))[1]$.
Thus
$\omega$ has \emph{even} degree when appearing as argument to $\tilde{\Uscr}$.
However the $\omega_\alpha$ as argument to $\tilde{\Uscr}$ have \emph{odd} degree.
By contrast the degree of the $\eta_\alpha$ is unchanged. 

We need to enumerate the graphs contributing to the evaluation of 
$\Uscr_{j+1}(\omega_{\alpha_1}\cdots \omega_{\alpha_j} \gamma)$. 
We need to consider
the following type of graphs.
\begin{enumerate}
\item There are $j$ vertices labeled by
  $\omega_{\alpha_1},\ldots,\omega_{\alpha_j}$. These have $1$
  outgoing arrow.
\item There is $1$ vertex labeled $\gamma$ which has $p+1$ outgoing arrows.
\item There are $m=p-j+1$ vertices labeled by elements $f_1,\ldots,f_m\in F$. 
\end{enumerate}
The edges leaving $\gamma$ are ordered by their ending vertex where we extend
the implied ordering of vertices of the first kind to all vertices via
$\omega_{\alpha_1}<\cdots<\omega_{\alpha_j}<\gamma<f_1<\cdots < f_m$.

Since there are no loops and double edges we find that $\gamma$ is
connected through an outgoing arrow with all other vertices. It
remains to allocate the $j$ arrows emanating from the vertices labeled
$\omega_\alpha$.

Recall that
by  [Ko, \S7.3.1.1 \S7.3.3.1]
that $W_\Gamma$ is zero if $\Gamma$ contains one the following subgraphs.
\begin{equation}
\label{ref-9.2-118}
\psfrag{p}[][]{$p$}
\psfrag{q}[][]{$q$}
\psfrag{r}[][]{$r$}
\includegraphics[width=4cm]{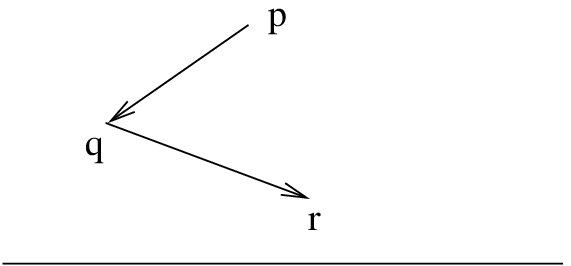} 
\end{equation}
or
\begin{equation}
\label{ref-9.3-119}
\psfrag{p}[][]{$p$}
\psfrag{q}[][]{$q$}
\psfrag{r}[][]{$r$}
\includegraphics[width=4cm]{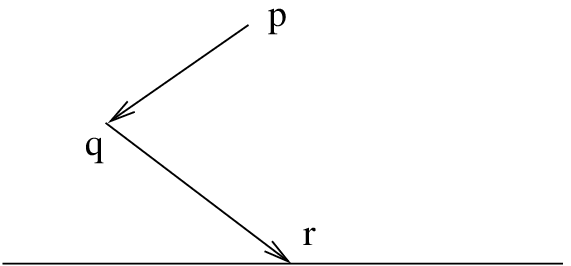}
\end{equation}
or
\begin{equation}
\label{ref-9.4-120}
\psfrag{p}[][]{$p$}
\psfrag{q}[][]{$q$}
\psfrag{r}[][]{$r$}
\includegraphics[width=4cm]{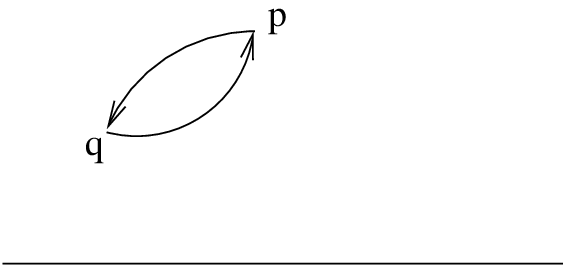}
\end{equation}
where $q$ has no additional incoming or outgoing vertices.

If there is an $\omega_\alpha$ which does not have an incoming arrow
from another $\omega_\alpha$ then we are in one of 
the situations \eqref{ref-9.2-118}\eqref{ref-9.3-119} or  \eqref{ref-9.4-120} and
hence $W_\Gamma=0$. The remaining graphs are of the form
\begin{equation}
\label{ref-9.5-121}
\psfrag{w1}[][]{$\omega_{\alpha_{\sigma(1)}}$}
\psfrag{w2}[][]{$\omega_{\alpha_{\sigma(2)}}$}
\psfrag{w3}[][]{$\omega_{\alpha_{\sigma(3)}}$}
\psfrag{w-2}[][]{$\omega_{\alpha_{\sigma(j-2)}}$}
\psfrag{w-1}[][]{$\omega_{\alpha_{\sigma(j-1)}}$}
\psfrag{w}[][]{$\omega_{\alpha_{\sigma(j)}}$}
\psfrag{g}[][]{$\gamma$}
\psfrag{f1}[][]{$f_1$}
\psfrag{fm}[][]{$f_m$}
\includegraphics[width=7cm]{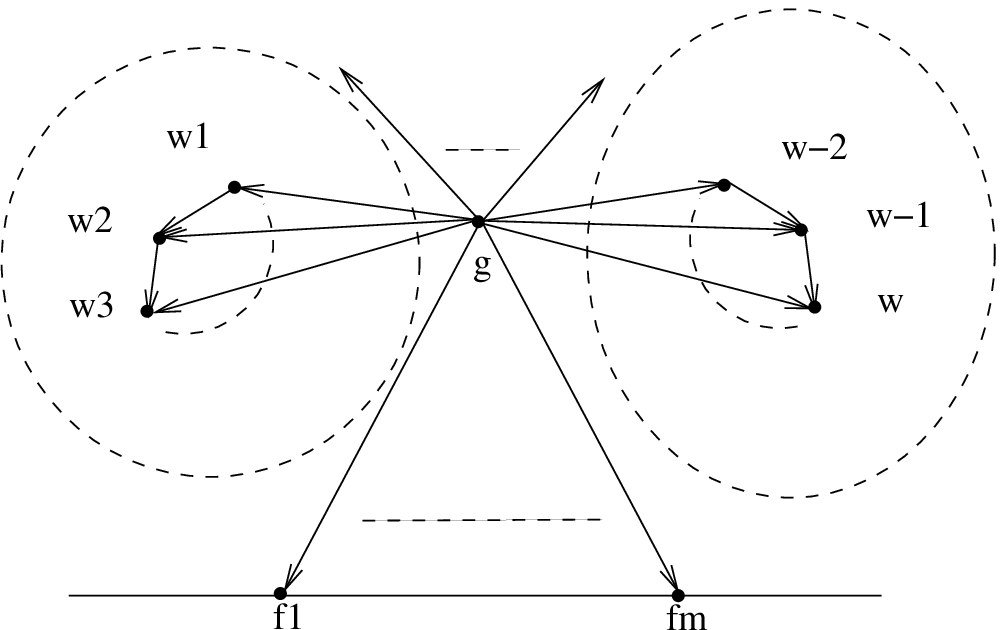}
\end{equation}
where $\sigma$ is a permutation of $\{1,\ldots,j\}$. 

Let us define $\Sigma_k$ as the following ``opposite wheel''
\begin{center}
\psfrag{1}[][]{$1$}
\psfrag{2}[][]{$2$}
\psfrag{k-1}[][]{$k-1$}
\psfrag{k}[][]{$k$}
\psfrag{k1}[][]{$k+1$}
\includegraphics[width=5cm]{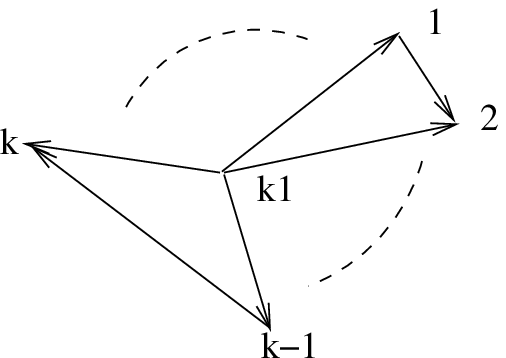}
\end{center}
and put $W_k=W_{\Sigma_k}$. To fix the sign we order the vertices  according to
their labels and the outgoing edges of the central vertex 
according to their ending vertex.

Now we compute $W_\Gamma$ and $\Uscr_\Gamma$ for a graph as in
\eqref{ref-9.5-121}. First we consider $W_\Gamma$. Assume that there are
$s$ wheels of size $l_1,\ldots, l_s$ respectively.

Write $g_i$ for the the edge emanating in $\omega_{\alpha_i}$ for
$i=1,\ldots,j$ and $e_i$ for the edge connecting $\gamma$ to
$\omega_{\alpha_i}$ for $i=1,\dots,j$. Finally write $h_i$ for the
edge connecting $\gamma$ to $f_i$ for $i=1,\ldots,p+1-j=m$. Then
\[
W_\Gamma=(-1)^{(m+2j)(m+2j-1)/2}\int d\phi_{g_1}\cdots d\phi_{g_j} d\phi_{e_1}\cdots d\phi_{e_j}d\phi_{h_1}\cdots d\phi_{h_{p+1-j}}
\]
To evaluate the integral we may put $\gamma$ in $i\in \Hscr$. This
reduces the symmetry group $G^{(1)}$ to the identity. We may clearly
choose the $\phi_{h_i}$ freely apart from the fact that
$\phi_{h_1}<\cdots < \phi_{h_{p+1-j}}$.  Thus we get
\begin{align*}
  W_\Gamma&=(-1)^{(m+2j)(m+2j-1)/2}\frac{1}{m!} 
\int  d\phi_{g_1}\cdots d\phi_{g_j} d\phi_{e_1}\cdots d\phi_{e_j}\\
  &=(-1)^{(m+2j)(m+2j-1)/2}\frac{1}{m!}\int d\phi_{g_{\sigma(1)}}\cdots
  d\phi_{g_{\sigma(j)}} d\phi_{e_{\sigma(1)}}\cdots d\phi_{e_{\sigma(j)}}\\
&=(-1)^{\sum_{p<q} l_pl_q} (-1)^{(m+2j)(m+2j-1)/2}\frac{1}{m!}
\int
d\phi_{g_{\sigma(1)}}\cdots d\phi_{g_{\sigma(l_1)}} d\phi_{e_{\sigma(1)}}\cdots d\phi_{e_{\sigma(l_1)}}\cdots\\
&=(-1)^{\sum_{p<q} l_pl_q}(-1)^{(m+2j)(m+2j-1)/2}(-1)^{\sum_i 2l_i(2l_i-1)/2}\frac{1}{m!} W_{l_1}\cdots W_{l_s}\\
&=(-1)^{\sum_{p<q} l_pl_q}(-1)^{(m+2j)(m+2j-1)/2}(-1)^j\frac{1}{m!}  W_{l_1}\cdots W_{l_s}
\end{align*}
where in the last line we have used the identities $(2l(2l-1))/2\equiv l\mod 2$ and $\sum_i l_i=j$. 

Now we compute\ $\Uscr_\Gamma(\omega_{\alpha_1}\cdots \omega_{\alpha_j}
\gamma)(f_1,\ldots,f_m)$. We are short of symbols so we use the same
symbol for an edge and for its corresponding index. Of course we use
the same ordering of the edges as above. We find
\begin{multline*}
\Uscr_\Gamma(\omega_{\alpha_1}\cdots \omega_{\alpha_j}
\gamma)(f_1,\ldots,f_m)=\\
(\partial_{e_{\sigma(1)}}\partial_{g_{\sigma(l_1)}} \omega^{g_{\sigma(1)}}_{\alpha_{\sigma(1)}})
(\partial_{e_{\sigma(2)}}\partial_{g_{\sigma(1)}} \omega^{g_{\sigma(2)}}_{\alpha_{\sigma(2)}})
(\partial_{e_{\sigma(3)}}\partial_{g_{\sigma(2)}} \omega^{g_{\sigma(3)}}_{\alpha_{\sigma(3)}})
\cdots
(\partial_{h_1} f_1\cdots \partial_{h_m} f_m)
\gamma^{e_1\cdots e_j h_1\cdots h_m}
\end{multline*}
We need a more concise way of writing this. Let $\Xi_\alpha$ be the
matrix of $1$-forms
$d(\partial_i\omega^j_\alpha)=\partial_k\partial_i(\omega^j_\alpha)
dt^k$. Then $\Uscr_\Gamma(\omega_{\alpha_1}\cdots \omega_{\alpha_j}
\gamma)(f_1,\ldots,f_m)$ is equal to
\def\HKR{\operatorname{HKR}}
\begin{multline*}
(-1)^\sigma\langle \Tr(\Xi_{\alpha_{\sigma(1)}}\cdots \Xi_{\alpha_{\sigma(l_1)}})\cdots
\Tr(\Xi_{\alpha_{\sigma(l_1+\cdots+l_{s-1}+1)}}\cdots \Xi_{\alpha_{\sigma(l_1+\cdots+l_{s-1}+l_s)}})
df_1\cdots df_m,\gamma\rangle\\
=(-1)^\sigma \langle df_1\cdots df_m, 
\Tr(\Xi_{\alpha_{\sigma(  l_1+\cdots+l_{s-1}+l_s    )}}
\cdots \Xi_{\alpha_{\sigma(    l_1+\cdots+l_{s-1}+1                              )}})
 \cdots
\Tr(\Xi_{\alpha_{\sigma(l_1)}}\cdots \Xi_{\alpha_{\sigma(1)}})
\wedge
\gamma\rangle\\
=(-1)^\sigma (-1)^{\frac{m(m-1)}{2}}m!\HKR(\Tr(\Xi_{\alpha_{\sigma(  l_1+\cdots+l_{s-1}+l_s    )}}
\cdots \Xi_{\alpha_{\sigma(    l_1+\cdots+l_{s-1}+1                              )}})
 \cdots
\Tr(\Xi_{\alpha_{\sigma(l_1)}}\cdots \Xi_{\alpha_{\sigma(1)}})
\wedge
\gamma)\\(f_1,\ldots,f_m)
\end{multline*}
where the first equality follows from \eqref{eqadjoint} and the second
equality is \eqref{ref-8.2-111}. So
\begin{multline*}
\Uscr_\Gamma(\omega_{\alpha_1}\cdots \omega_{\alpha_j}
\gamma)=\\(-1)^\sigma (-1)^{\frac{m(m-1)}{2}}m!\HKR(\Tr(\Xi_{\alpha_{\sigma(  l_1+\cdots+l_{s-1}+l_s    )}}
\cdots \Xi_{\alpha_{\sigma(    l_1+\cdots+l_{s-1}+1                              )}})
 \cdots
\Tr(\Xi_{\alpha_{\sigma(l_1)}}\cdots \Xi_{\alpha_{\sigma(1)}})
\wedge\gamma)
\end{multline*}
Put
\[
\tilde{\Uscr}_\Gamma(\omega^j\gamma)(f_1,\ldots,f_n)= \eta_{\alpha_j}\cdots \eta_{\alpha_1}\Uscr_\Gamma(\omega_{\alpha_1}\cdots \omega_{\alpha_j}
\gamma)(f_1,\ldots,f_m)
\]
An easy computation yields
\[
\tilde{\Uscr}_\Gamma(\omega^j \gamma)=
(-1)^{j(j-1)/2}(-1)^{m(m-1)/2}
m!\HKR(\Tr(\Xi^{l_s}) \cdots\Tr(\Xi^{l_1})
\wedge\gamma)
\]
where we have extended $-\wedge-$ and $\HKR(-)$ to operations over $C^{\coord,L}$
and where $\Xi$ is the matrix $\eta_\alpha d(\partial_i\omega^j_\alpha)$ of 
elements of $C^{\coord}_1\ctimes \Omega^1_F$.  The entries of $\Xi$
have even total degree so the traces $\Tr(\Xi^l)$  commute. 

Now note the following simple identities
\[
\frac{j(j-1)}{2}=\frac{(\sum_i l_i)(\sum_i l_i-1)}{2}=\sum_{i<j} l_il_j
+\sum_i \frac{l_i(l_i-1)}{2}
\]
\[
\frac{(m+2j)(m+2j-1)}{2}=
\frac{m(m-1)}{2}+j\mod 2
\]
Collecting all signs we deduce
\[
W_\Gamma\tilde{\Uscr}_\Gamma(\omega^j \gamma)=
(-1)^{\sum_i l_i(l_i-1)/2}W_{l_1}\cdots W_{l_s}\HKR(\Tr(\Xi^{l_1})\cdots
\Tr(\Xi^{l_s})
\wedge\gamma)
\]
Putting temporarily 
\[
X_l=(-1)^{l(l-1)/2}W_l \Tr(\Xi^l)
\]
we find
\[
W_\Gamma\tilde{\Uscr}_\Gamma(\omega^j \gamma)=
\HKR( X_{l_1}\cdots
X_{l_s}
\wedge\gamma\rangle)
\]
Now we have to enumerate the number of possible graphs $\Gamma$.
Ordering the size of the wheels in increasing order we get a partition
$\tau=(1\cdots 1\cdots r\cdots r)$ where $i$ occurs $\tau_i$ times.
The number distinct graphs corresponding to such a partition is
\[
\frac{j!}{ \tau_1!\cdots \tau_r!\,1^{\tau_1}\cdots r^{\tau_r} }
\] 
Thus we find that
\[
\sum_{j\ge 0} \frac{1}{j!} \tilde{\Uscr}_{j+1}
(\omega^j\gamma)=\sum_{\tau_1,\tau_2,\ldots}
\frac{1}{ \tau_1!\tau_2!\cdots 1^{\tau_1}2^{\tau_2}\cdots }\HKR( X_{1}^{\tau_1}X_{2}^{\tau_2}\cdots 
\wedge\gamma)
\]
Formally we have
\[
e^{X_r/r}=\sum_{\tau_r} \frac{1}{ \tau_r!r^{\tau_r }} X^{\tau_r}_r
\]
so that we find
\[
\sum_{j\ge 0} \frac{1}{j!} \tilde{\Uscr}_{j+1}
(\omega^j\gamma)=
\HKR( e^{X_1+X_2/2+\cdots} \wedge\gamma)
\]
So if we put
\[
\Theta=\sum_l (-1)^{l(l-1)/2} \frac{1}{l}W_l \Xi^l
\]
then 
\begin{align*}
  \sum_{j\ge 0} \frac{1}{j!} \tilde{\Uscr}_{j+1}
  (\omega^j\gamma)&=
  \HKR(e^{\Tr(\Theta)} \wedge\gamma)\\
  &=\HKR(\det (e^{\Theta}) \wedge\gamma)
\end{align*}
This formula was proved under the assumption that $\gamma\in T_{\poly}(F)$.
However by linear extension it follows that it remains true if
$\gamma \in C^{\coord,L}\ctimes T_{\poly}(F)$. Thus our final formula is
\begin{equation}
\label{ref-9.6-122}
\tilde{\Uscr}_{\omega,1}=\HKR(\det (e^{\Theta}) \wedge -)
\end{equation}

\medskip

We now analyze the series $\Theta$ is more detail. We need to know the
value of $W_l$. As explained to us by Torossian this can be obtained
from the work of Cattaneo and Felder on the quantization of
coisotropic submanifolds \cite{CF2,CF3}.  See \cite[Thm
18]{Willwacher}.  As an alternative one can use a tedious but
elementary computation using Stokes theorem
\cite[(1.1)]{wheelweight}. The result is the following.
\begin{lemmas} 
\label{weightlemma} We have
\begin{equation}
\label{ref-9.7-123}
W_l=-(-1)^{(l+1)l/2}l\ss_l
\end{equation}
where $\ss_n$ is the $n$'th modified Bernouilli number $\ss_n$ which is defined
by
\[
\sum_l \ss_l x^l=\frac{1}{2}\log \frac{e^{x/2}-e^{-x/2}}{x}
\]
\end{lemmas}
Let us finally also mention that a result similar to
\eqref{ref-9.7-123} was announced by Shoikhet
in \cite[\S2.3.1]{Shoikhet}.  It can presumably be obtained from the
methods in \cite{Shoikhet1}.

Substituting we find
\begin{align*}
\Theta&=-\sum_l (-1)^l \ss_l \Xi^l\\
&=-\frac{1}{2}\log \frac{e^{\Xi/2}-e^{-\Xi/2}}{\Xi}
\end{align*}
and hence
\begin{equation}
\label{ref-9.8-124}
e^{\Theta}=\sqrt{\frac{\Xi}{e^{\Xi/2}-e^{-\Xi/2}}}
\end{equation}
It will be convenient for a module $N$ with a connection to introduce the
modified Todd class as follows
\[
\widetilde{\td}(N)=\det(\tilde{q}(A(N)))
\]
where

\[
\tilde{q}(x)=\frac{x}{e^{x/2}-e^{-x/2}}
\]

It follows from Proposition \ref{ref-7.4.2-107} that the following diagram
is commutative. 
\[
\begin{CD}
T_{\poly,C^{\aff,L}}(C^{\aff,L}\ctimes_{R_1} JL) @>>> C^{\coord,L}\ctimes T_{\poly}(F)\\
@V \widetilde{\td}(\Nscr^{\aff})^{1/2}\wedge-VV @VV \det e^\Theta\wedge-V\\
T_{\poly,C^{\aff,L}}(C^{\aff,L}\ctimes_{R_1} JL) @>>> C^{\coord,L}\ctimes T_{\poly}(F)
\end{CD}
\]
Thus we get 
\begin{equation}
\label{ref-9.9-125}
\Vscr^{\aff}_1=\HKR\circ(\widetilde{\Td}(\Nscr^{\aff})^{1/2}\wedge -)
\end{equation}

\subsection{The global case}
\label{ref-9.2-126}
Now we globalize things. Let $(\Cscr,\Oscr,\Lscr)$ be as in
\S\ref{ref-6.4-88}. Then it follows from Proposition
\ref{ref-7.4.2-107} (specifically \eqref{ref-7.10-108}) that the following diagram is commutative in
$D(\Oscr)$.
\[
\begin{CD}
  T^{\Lscr}_{\poly}(\Oscr)@>\cong>>
  T_{\poly,C^{\aff,\Lscr}}(C^{\aff,L}\ctimes_{\Oscr_1} J\Lscr)\\
@V \widetilde{\Td}(\Lscr)^{1/2}\wedge- VV @VV \widetilde{\Td}(\Nscr^{\aff})^{1/2}\wedge- V\\
T^{\Lscr}_{\poly}(\Oscr)@>>\cong>
  T_{\poly,C^{\aff,\Lscr}}(C^{\aff,L}\ctimes_{\Oscr_1} J\Lscr)
\end{CD}
\]
so that we get a commutative diagram in $D(\Oscr)$.
\[
\begin{CD}
  T^{\Lscr}_{\poly}(\Oscr)@>\cong>>
  T_{\poly,C^{\aff,\Lscr}}(C^{\aff,L}\ctimes_{\Oscr_1} J\Lscr)\\
@V\HKR(\widetilde{\Td}(\Lscr)^{1/2}\wedge-) VV @VV  \Vscr^{\aff}_1 V\\
D^{\Lscr}_{\poly}(\Oscr)@>>\cong>
  D_{\poly,C^{\aff,\Lscr}}(C^{\aff,L}\ctimes_{\Oscr_1} J\Lscr)
\end{CD}
\]
where $\widetilde{\Td}(\Lscr)^{1/2}$ is as defined in the introduction. Since the
horizontal isomorphisms as well as $\Vscr_1^{\aff}$ are Gerstenhaber
algebra morphisms (Lemma \ref{ref-9.1.1-117}) the same holds for
$\HKR(\widetilde{\Td}(\Lscr)^{1/2}\wedge-)$ as well.  This finishes the proof of Theorem
\ref{ref-1.3-0} in the case $k$ contains the reals and the Todd class is
replaced by the modified Todd class.

\subsection{Proof for the ordinary Todd class}
We have
\begin{align*}
\tilde{\Td}(\Lscr)&= \Td(\Lscr)\det(e^{-A(\Lscr)/2})\\
&=\Td(\Lscr)e^{-\Tr(A(\Lscr))/2}\\
&=\Td(\Lscr)e^{-a_1(\Lscr)/2}
\end{align*}
In other words it is sufficient to prove that $e^{-a_1(\Lscr)/4}\wedge-$ 
defines an automorphism of $T^{\Lscr}_{\poly}(\Oscr)$ as a Gerstenhaber algebra
in $D(X)$. 

We may as well prove that
$e^{-\Tr(\Xi)/4}\wedge-$ is compatible with the Lie bracket and the
cupproduct on $C^{\coord}\ctimes T_{\poly}(F)$ or equivalently that
$\Tr(\Xi)\wedge-$ is a derivation for these operations.  We have
$\Tr(\Xi)=\sum_{i,\alpha}\eta_\alpha d(\partial_i\omega^i_\alpha)$.
Put $b_\alpha=\sum_i\partial_i\omega^i_\alpha$. Since everying is
$C^{\coord,\Lscr}$ linear it is suffient to prove that $db_\alpha\wedge-$ is
a derivation for the Lie algebra and cupproduct on $T_{\poly}(F)$. 

Since this fact is clear for the cupproduct we only look at the Lie bracket.
For $D,E\in T^0_{\poly}(F)$ we have 
\begin{align*}
db_\alpha\wedge [D,E]&=[D,E](b_\alpha)\\
&=D(E(b_{\alpha})-E(D(b_\alpha))\\
&=D(db_{\alpha}\wedge E)-E(db_{\alpha}\wedge D )\\
&=[D,db_{\alpha}\wedge E]+[db_{\alpha}\wedge D,E]
\end{align*}
finishing the proof.
\subsection{Arbitrary base fields}
Now we let $k$ be arbitrary (of characteristic zero) and we choose an
embedding $k\subset \CC$.  It follows from the formulas \eqref{ref-9.6-122} and
\eqref{ref-9.8-124} that $\tilde{\Uscr}_{\omega,1}$, while initially defined
over $\CC$ ($\RR$ in fact), actually descends to $k$.  We will denote
the descended morphism by $u$.

Now note the following. 
\begin{propositions}
\begin{enumerate}
\item There exist $w_\Gamma\in k$ for $\Gamma\in G_{m,n}$, which are
zero when $W_\Gamma$ is zero,  such that for
\[
l(\alpha,\beta)=\sum_{n,m\ge 0, \Gamma\in G_{n,m}} w_\Gamma\Uscr_\Gamma(\alpha\beta
\omega^{n-2})
\]
we have
\begin{equation}
\label{ref-9.10-127}
[u(\alpha),u(\beta)]
-u([\alpha,\beta])
+d(l(\alpha,\beta))-l(d\alpha,\beta)-(-1)^{|\alpha|}l(\alpha,d\beta)=0
\end{equation}
\item
There exist $\tilde{w}_\Gamma\in k$ for $\Gamma\in G_{m,n,1}$, which are zero
when $\tilde{W}_\Gamma$ is zero, such that for
\[
h(\alpha,\beta)=\sum_{n,m\ge 0, \Gamma\in G_{n,m,1}} \tilde{w}_\Gamma\Uscr_\Gamma(\alpha\beta
\omega^{n-2})
\]
we have
\begin{equation}
\label{ref-9.11-128}
u(\alpha)\cup u(\beta)
-u(\alpha\cup\beta)
+d(h(\alpha,\beta))-h(d\alpha,\beta)-(-1)^{|\alpha|}h(\alpha,d\beta)=0
\end{equation}
\end{enumerate}
\end{propositions}
\begin{proof} The equations \eqref{ref-9.10-127}\eqref{ref-9.11-128} are linear in $w_\Gamma$,
$\tilde{w}_\Gamma$. By  the fact that  $\tilde{\Uscr}_\omega$ 
is a $L_\infty$-morphism and Proposition \ref{ref-8.1-112}
there is a solution over $\CC$.
Hence there is
a solution over $k$ (for example obtained by applying an arbitrary projection
$\CC\r k$). 
\end{proof}
We can now proceed as in \S\ref{ref-9.1-115},\ref{ref-9.2-126}
(using an analogue of Lemma \ref{ref-9.1.1-117}) to construct
a commutative diagram
\[
\xymatrix{ \scriptstyle T_{\poly}^{\Lscr_2}(\Oscr_2)\ar[r] &
  \scriptstyle
  T_{\poly,C^{\aff,\Lscr}}(C^{\aff,\Lscr}\ctimes_{\Oscr_1}
  J\Lscr)\ar[r] \ar[d]^{v^{\aff}} & \scriptstyle
  T_{\poly,C^{\coord,\Lscr}}(C^{\coord,\Lscr}\ctimes_{\Oscr_1} J\Lscr)
  \ar[r]^-\cong\ar[d]^{v^{\coord}}& \scriptstyle
  (C^{\coord,\Lscr}\ctimes T_{\poly}(F))_{\omega}
  \ar[d]^{u}\\
  \scriptstyle D_{\poly}^{\Lscr_2}(\Oscr_2)\ar[r] & \scriptstyle
  D_{\poly,C^{\aff,\Lscr}}(C^{\aff,\Lscr}\ctimes_{\Oscr_1}
  J\Lscr)\ar[r] & \scriptstyle
  D_{\poly,C^{\coord,\Lscr}}(C^{\coord,\Lscr}\ctimes_{\Oscr_1} J\Lscr)
  \ar[r]_-\cong& \scriptstyle (C^{\coord,\Lscr} \ctimes
  D_{\poly}(F))_{\omega} }\]
where $v^{\aff}$ commutes both with the Lie bracket and the cupproduct
up to a global homotopy. Using the fact that the formula \eqref{ref-9.6-122} 
continues to hold
\[
u=\HKR\circ(\det(\Theta)\wedge -)
\] 
we can now continue as in \S\ref{ref-9.2-126} to finish the proof of Theorem
\ref{ref-1.3-0}. 
\def\cprime{$'$} \def\cprime{$'$} \def\cprime{$'$}
\ifx\undefined\bysame
\newcommand{\bysame}{\leavevmode\hbox to3em{\hrulefill}\,}
\fi

\end{document}